\documentclass{amsart}

\usepackage{amssymb}
 \usepackage{amsthm}

\usepackage{researchmacros}
\usepackage{ytableau}
\usepackage{mathtools}
\usepackage{enumitem}
\usepackage{tikz}
\usetikzlibrary{cd}
\usepackage[all]{xy}
\usepackage[vcentermath]{youngtab}

\usepackage{fullpage}

\usepackage{pifont}

\newtheorem{theorem}{Theorem}[section] 
\newtheorem{lemma}[theorem]{Lemma}     
\newtheorem{corollary}[theorem]{Corollary}
\newtheorem{proposition}[theorem]{Proposition}

\theoremstyle{definition}
\newtheorem{definition}[theorem]{Definition}
\newtheorem{remark}[theorem]{Remark}

\newtheorem{question}[theorem]{Question}
\newtheorem{example}[theorem]{Example}

\newcommand{\st}{\,|\,}
\newcommand{\vspan}{\mathrm{span}}

\newcommand{\Fl}{\mathrm{Fl}}
\newcommand{\GL}{\mathrm{GL}}
\newcommand{\Hilb}{\mathrm{Hilb}}
\newcommand{\Frob}{\mathrm{Frob}}
\newcommand{\la}{\lambda}
\newcommand{\im}{\mathrm{im}}
\newcommand{\SYT}{\mathrm{SYT}}

\newcommand{\fl}{\mathrm{f\hspace{0.25pt}l}}

\newcommand{\inv}{\mathrm{inv}}

\newcommand{\sh}{\mathrm{sh}}
\newcommand{\bx}{\mathbf{x}}
\newcommand{\fgl}{\mathfrak{gl}}
\newcommand{\ft}{\mathfrak{t}}
\newcommand{\rk}{\mathrm{rk}}

\begin{document}




\title{Springer fibers and the Delta Conjecture at $t=0$}

%
%
%
%
%
%
%

\author{Sean T. Griffin}
\thanks{Sean T. Griffin was partially supported by NSF Grant DMS-1439786 while in residence at the Institute for Computational and Experimental Research in Mathematics in Providence, RI, during the Spring 2021 semester.}
\address{Department of Mathematics, University of California Davis, Davis, CA, USA}
\email{stgriffin@ucdavis.edu}

\author{Jake Levinson}
\thanks{Jake Levinson was partially supported by an AMS Simons Travel Grant.}
\address{Department of Mathematics, Simon Fraser University, Burnaby, BC, Canada}
\email{jake\_levinson@sfu.ca}

\author{Alexander Woo}
\thanks{Alexander Woo was partially supported by Simons Collaboration Grant 359792}
\address{Department of Mathematics and Statistical Science, University of Idaho, Moscow, ID, USA}
\email{awoo@uidaho.edu}

\begin{abstract}
  We introduce a family of varieties $Y_{n,\lambda,s}$, which we call
  the \emph{$\Delta$-Springer varieties}, that generalize the type A
  Springer fibers. We give an explicit presentation of the cohomology
  ring $H^*(Y_{n,\lambda,s})$ and show that there is a symmetric group
  action on this ring generalizing the Springer action on the
  cohomology of a Springer fiber. In particular, the top cohomology
  groups are induction products of Specht modules with trivial modules. The $\lambda=(1^k)$ case of this construction gives a compact geometric realization for the expression in the Delta Conjecture at $t=0$. Finally, we generalize results of De Concini and Procesi on the scheme of diagonal nilpotent matrices by  constructing an ind-variety $Y_{n,\lambda}$ whose cohomology ring is isomorphic to the coordinate ring of the scheme-theoretic intersection of an Eisenbud--Saltman rank variety and diagonal matrices.
\end{abstract}


%
%
%
%
%
%
%



\maketitle

\section{Introduction}\label{sec:Introduction}
In this article, we introduce a family of varieties generalizing the Springer fibers, which we call the {\bf $\Delta$-Springer varieties}. We establish an explicit presentation of the cohomology ring of a $\Delta$-Springer variety generalizing the one given by Tanisaki for the case of a Springer fiber, showing that this cohomology ring is the ring $R_{n,\lambda,s}$ introduced by the first author~\cite{GriffinOSP}. As a special case, our construction gives a new \emph{compact} geometric realization of the  expression in the Delta Conjecture at $t=0$. We also prove a version of the Springer correspondence for this family of varieties, showing that the top cohomology group has the $S_n$-module structure of an induction product of a Specht module and a trivial module. Finally, we generalize work of De Concini and Procesi~\cite{dCP} by introducing an ind-variety whose cohomology ring coincides with the coordinate ring of the scheme-theoretic intersection of an Eisenbud--Saltman rank variety with diagonal matrices. This is the full version of the extended abstract~\cite{GLW-FPSAC}.

\subsection{Background and context}

In the seminal work~\cite{Springer-TrigSum,Springer-WeylGrpReps}, Springer introduced a family of varieties associated to any complete flag variety $G/B$, now called Springer fibers, that have remarkable connections to the representation theory of Weyl groups. 
Springer constructed a nontrivial action of the Weyl group on the cohomology ring of a Springer fiber that does not come from any action on the Springer fiber itself. In particular, in type A, where the Weyl group is the symmetric group $S_n$, Springer proved the \textbf{Springer correspondence}, which states that the highest degree nonzero cohomology group of a Springer fiber is an irreducible representation of $S_n$, and every finite dimensional irreducible representation appears this way. Since Springer's original work, several other geometric constructions of Springer's representation (up to tensoring with the sign representation) have been discovered~\cite{CG,Lusztig,Slodowy}.

The graded $S_n$-module type of the entire cohomology ring of a Springer fiber was computed by Hotta and Springer~\cite{Hotta-Springer}. Under the Frobenius characteristic map $\Frob$ that associates a symmetric function to each $S_n$-module, the cohomology ring of a Springer fiber is sent to the \textbf{modified Hall--Littlewood symmetric function} 
\begin{align}
\Frob(H^*(\cB^\lambda;\bQ);q) = \widetilde H_\lambda(\bx;q^2),
\end{align}
where the $q$ on the left-hand side keeps track of the grading of the cohomology ring.

De Concini and Procesi~\cite{dCP} and Tanisaki~\cite{Tanisaki} found an explicit presentation for the graded ring $H^*(\cB^\lambda;\bQ)$ as a quotient $R_\lambda$ of the polynomial ring $\bQ[x_1,\dots, x_n]$. The $S_n$-module structure on $H^*(\cB^\lambda;\bQ)$ is compatible with this presentation in the sense that $S_n$ acts on the ring $R_\lambda$ by permuting the variables $x_1,\dots, x_n$, and the isomorphism between the two rings is $S_n$ equivariant. A detailed analysis of the connection between this presentation $R_\lambda$ and modified Hall--Littlewood symmetric functions was given by Garsia and Procesi~\cite{Garsia-Procesi}. Springer fibers and their combinatorial and representation-theoretic properties served as motivation for Haiman's work connecting Macdonald polynomials, a generalization of Hall--Littlewood symmetric functions, with the geometry of Hilbert schemes~\cite{Haiman01,Haiman02}.

On another related line of research, the Delta Conjecture of Haglund, Remmel, and Wilson~\cite{HRW} predicts two combinatorial formulas for a particular symmetric function $\Delta'_{e_{k-1}} e_n(q, t)$ with $q$ and $t$ parameters coming from the theory of Macdonald polynomials. The ``rise formula'' half of the conjecture has been recently proven independently by D'Adderio and Mellit~\cite{DM}  and by Blasiak, Haiman, Morse, Pun, and Seelinger~\cite{BHMPS}, and the two works prove different generalizations of the rise formula. 

Since $\Delta'_{e_{k-1}} e_n$ is also conjectured to be Schur-positive, there is much interest in a natural algebraic or geometric construction of a (bigraded) $S_n$-module whose Frobenius characteristic is $\Delta'_{e_{k-1}} e_n$. Zabrocki has conjectured an algebraic formulation in the general case~\cite{Zabrocki}.

In the $t=0$ case of the conjecture, Haglund, Rhoades, and Shimozono~\cite{HRS1} found and proved an algebraic realization by constructing a graded ring $R_{n,k}$ with a suitable $S_n$-action whose graded Frobenius characteristic is $\Delta'_{e_{k-1}} e_n(q, 0)$ (after tensoring by the sign representation and reversing the grading).
A parallel geometric interpretation was given by Pawlowski and Rhoades~\cite{Pawlowski-Rhoades}, who discovered a complex algebraic variety whose cohomology ring is $R_{n,k}$. 
We note that the Hilbert--Poincar\'e series of $R_{n,k}$, which is determined by $\Delta'_{e_{k-1}}e_n(q,0)$, is not palindromic in $q$. Therefore, any such variety must be either non-compact or singular by Poincar\'e duality.
Pawlowski and Rhoades defined the non-compact smooth space of \textbf{spanning line arrangements}, $n$-tuples of lines in $\bC^k$ that span $\bC^k$,
\begin{align}
X_{n,k} \coloneqq \{(L_1,\dots, L_n) \in (\bP^{k-1})^n \st L_1+\cdots +L_n = \bC^k\}.
\end{align}
They proved that there is an isomorphism of graded rings and graded $S_n$-modules
\begin{align}
H^*(X_{n,k}) \cong R_{n,k},
\end{align}
thus giving a connection between the expression in the Delta Conjecture at $t=0$ and geometry.

\subsection{Results of this paper}

In this article, we introduce a compact and singular variety $Y_{n,(1^k),k}$ whose cohomology ring is the Haglund--Rhoades--Shimozono ring $R_{n,k}$. Thus, the variety $Y_{n,(1^k),k}$ gives a new geometric realization of $\Delta'_{e_{k-1}} e_n(q, 0)$. Furthermore, we situate $Y_{n,(1^k),k}$ as part of a broader family of subvarieties $Y_{n,\lambda,s}$ of the partial flag variety $\Fl_{(1^n)}(\bC^K)$, where $K=n+(s-1)(n-|\lambda|)$.
When $n=|\lambda|$, the $\Delta$-Springer variety $Y_{n,\lambda,s}$ is the Springer fiber $\cB^\lambda$ for all $s\geq \ell(\la)$.
Hence we call the varieties $Y_{n,\la,s}$ the \textbf{$\Delta$-Springer varieties}.  We then use techniques from the study of Springer fibers to study $Y_{n,\la,s}$. In particular, we show that $Y_{n,\la,s}$ is simultaneously a subvariety of a certain Steinberg variety and the projected image of a certain Spaltenstein variety. We use these facts to compute $H^*(Y_{n,\la,s})$. Our work situates the study of $R_{n,k}$ and $\Delta'_{e_{k-1}}e_n$ in the context of the theory of Springer fibers and geometric representation theory.

As our main result, we give and prove an explicit presentation of the ring $H^*(Y_{n,\la,s})$ as a quotient of a polynomial ring, generalizing Tanisaki's presentation for the cohomology ring of a Springer fiber~\cite{Tanisaki}. This presentation coincides with that of the graded ring $R_{n,\lambda,s}$ recently introduced by the first author~\cite{GriffinOSP}.

\begin{theorem}\label{thm:MainThmIntro}
We have $H^*(Y_{n,\la,s}) \cong R_{n,\la,s}$ as graded rings.
\end{theorem}

As a consequence, we see that the cohomology ring of $Y_{n,\lambda,s}$ has a graded $S_n$-module structure inherited from the action on $R_{n,\la,s}$ via permuting the variables $x_i$. This $S_n$-module structure on $H^*(Y_{n,\la,s})$ generalizes the action in the Springer fiber case, and it is compatible with an action on the cohomology of $H^*(\Fl_{(1^n)}(\bC^K))$ in the following sense. Let $k=|\lambda|$, and let $\alpha$ be the composition $\alpha = (1^n,(s-1)(n-k))$, where $1^n$ signifies that $\alpha$ begins with a string of $n$ many parts of size $1$. Denote by $S_\alpha = S_1\times \cdots \times S_1 \times S_{(s-1)(n-k)}$ the corresponding Young subgroup of $S_K$. We have that $S_n$ embeds into $S_K$ as the subgroup of permutations that fix $n+1,n+2,\dots, K$. Then $H^*(\Fl_{(1^n)}(\bC^K))$ has an $S_n$-module structure, and we have a commutative diagram of graded rings and graded $S_n$ modules,
\begin{equation}\label{eq:SnDiagram}
\begin{tikzcd}
\left(\dfrac{\bZ[x_1,\dots, x_K]}{\langle e_1,\dots, e_K\rangle}\right)^{S_{\alpha}} \arrow[r, "\cong"]\arrow[d,twoheadrightarrow] & H^*(\Fl_{(1^n)}(\bC^K))\arrow[d,twoheadrightarrow]\\
R_{n,\la,s} \arrow[r, "\cong"]& H^*(Y_{n,\la,s}),
\end{tikzcd}
\end{equation}
where the vertical maps are surjections, and the horizontal maps are isomorphisms. The superscript $S_\alpha$ in the diagram denotes the subring of (cosets of) polynomials invariant under the action of $S_\alpha$.

One of the known constructions of the Springer action, due to Borho and MacPherson, uses an action defined on perverse sheaves~\cite{Borho-MacPherson}.  This construction was adapted to Spaltenstein varieties by Brundan and Ostrik~\cite{Brundan-Ostrik},
and we believe it should be possible to give a direct construction of the $S_n$ action on $H^*(Y_{n,\la,s})$ along the same lines.

In order to prove Theorem~\ref{thm:MainThmIntro}, we construct an affine paving of the space $Y_{n,\la,s}$. As a corollary of the paving, we show that $Y_{n,\la,s}$ is equidimensional, and we compute its dimension. Furthermore, we give a characterization of the irreducible components of $Y_{n,\la,s}$ (Theorem~\ref{thm:irreducible-components}).
\begin{theorem}\label{thm:DimensionOfYIntro}
The variety $Y_{n,\lambda,s}$ is equidimensional of complex dimension
\begin{align}\label{eq:DimExpression}
    d = \sum_i \binom{\lambda_i'}{2} + (s-1)(n-k).
\end{align}
\end{theorem}

We use the graded Frobenius characteristic formulas in~\cite{GriffinOSP} for $R_{n,\la,s}$ to prove a generalization of the  Springer correspondence to the setting of induced Specht modules. 
\begin{theorem}\label{thm:GenSpringerCorrespondence}
Let $d$ be as in \eqref{eq:DimExpression}, and consider $S_k$ as the subgroup of $S_n$ fixing ${k+1},\dots, n$. For $s>\ell(\la)$, we have an isomorphism of $S_n$-modules
\begin{align}
H^{2d}(Y_{n,\la,s};\bQ) \cong \mathrm{Ind}\!\uparrow_{S_k\times S_{n-k}}^{S_n} (S^\la),
\end{align}
where $S_k\times S_{n-k}$ acts on $S^\la$ by its usual action of $S_k$
(and $S_{n-k}$ acts trivially). For $s=\ell(\la)$, letting $\Lambda = (n-k+\la_1,\dots, n-k+\la_s)$, we have
\begin{align}
    H^{2d}(Y_{n,\la,s};\bQ) \cong S^{\Lambda/(n-k)^{s-1}},
\end{align}
the Specht module of skew shape $\Lambda/(n-k)^{s-1}$.
\end{theorem}

Finally, we generalize results of De Concini and Procesi~\cite{dCP}. Let $\mathfrak{gl}_n$ be the Lie algebra of $n\times n$ matrices over $\bQ$. Given $\la\vdash n$, define $O_\lambda\subseteq \mathfrak{gl}_n$ to be the set of $n\times n$ nilpotent matrices over $\bQ$ with Jordan type $\lambda$, and let $\overline{O}_\lambda$ be its closure in $\mathfrak{gl}_n$. Let $\mathfrak{t}\subset \mathfrak{gl}_n$ be the Cartan subalgebra of diagonal matrices. Since the only nilpotent diagonal matrix is the zero matrix, the scheme-theoretic intersection $\overline{O}_{\la}\cap \mathfrak{t}$ is a scheme supported at a single point. Denoting the coordinate ring of this scheme by $\bQ[\overline{O}_{\la}\cap \mathfrak{t}]$, De Concini and Procesi proved that
\begin{align}
    H^*(\cB^\la;\bQ) \cong \bQ[\overline{O}_{\la'}\cap \mathfrak{t}].
\end{align}

We prove the following generalization of De Concini and Procesi's result. Given a partition $\lambda$ of size at most $n$, let $\overline{O}_{n,\la}$ be the Eisenbud--Saltman rank variety (defined in Section~\ref{sec:IndVariety}).
\begin{theorem}\label{thm:ES-Iso}
Define the ind-variety $Y_{n,\la} \coloneqq \bigcup_{s\geq \ell(\la)} Y_{n,\la,s}$. There is an isomorphism of graded rings and graded $S_n$-modules
\begin{align}\label{eq:dCPGeneralization}
    H^*(Y_{n,\la};\bQ) \cong \bQ[\overline{O}_{n,\la'}\cap \mathfrak{t}].
\end{align}
\end{theorem}

\subsection{Structure of the paper}

In Section~\ref{sec:Background}, we outline preliminary definitions and previous results. In Section~\ref{sec:AffinePaving}, we define the $\Delta$-Springer variety $Y_{n,\lambda,s}$ and prove that it has an affine paving. We prove that the cells of this paving have an inductive structure that allows us to compute the rank generating function of the cohomology ring.
In Section~\ref{sec:EmptyPartition}, we analyze the case $\lambda=\emptyset$ and prove that the variety $Y_{n,\emptyset,s}$ is an iterated projective bundle whose cohomology ring is the same as the cohomology ring of a product of projective spaces.
In Section~\ref{sec:SpaltensteinAndCohomology}, we show that $Y_{n,\la,s}$ is the image of a projection down from a Spaltenstein variety, and we use this to prove Theorem~\ref{thm:MainThmIntro}.
In Section~\ref{sec:IrreducibleComponents}, we prove Theorem~\ref{thm:DimensionOfYIntro} by characterizing the irreducible components of $Y_{n,\la,s}$, and we prove our generalization of the Springer correspondence, Theorem~\ref{thm:GenSpringerCorrespondence}.
In Section~\ref{sec:IndVariety}, we introduce $Y_{n,\la}$ and prove Theorem~\ref{thm:ES-Iso}.
Finally, in Section~\ref{sec:FutureWork} we list some open problems.

\section{Background}\label{sec:Background}

\subsection{Flag varieties and Schubert cells}\label{subsec:Flags}

Given a complex vector space $V=\bC^K$, a \textbf{partial flag} is a nested sequence of vector subspaces of $V$,
\begin{align}
    V_\bullet = (V_1\subset V_2\subset\dots\subset V_m).
\end{align}
Given a composition $\alpha = (\alpha_1,\dots, \alpha_m)$ with $\alpha_i>0$ for all $i$ and $\alpha_1+\cdots+\alpha_m \leq \dim(V)$,  define the \textbf{partial flag variety} to be the set of partial flags of $V$ such that the dimensions of the successive quotients $V_i/V_{i-1}$ are recorded by the parts of $\alpha$, so
\begin{align}
    \Fl_{\alpha}(V) \coloneqq \{V_\bullet = (V_1\subset\dots\subset V_m) \st  V_i\subset V,\, \dim(V_i/V_{i-1}) = \alpha_i\text{ for }i\leq m\}.
\end{align}
In the case when $K=n$ and $\alpha = (1^n)$, we recover the \textbf{complete flag variety}, denoted by $\Fl(n) = \Fl_{(1^n)}(\bC^n)$.  The (partial) flag variety is realized as a projective algebraic variety as $G/P_{\alpha}$, where $G=GL_K(\bC)$ and $P_\alpha$ is a parabolic subgroup determined by $\alpha$.

From here on, we focus on the case where $\alpha=(1^n)$ for some $n\leq K$.  Let $f_1,f_2,\dots,f_K$ be the standard ordered basis of $\bC^K$.  Given an injective map $w:[n]\rightarrow[K]$, let the {\bf coordinate flag} $F_\bullet^{(w)}$ be defined by setting $F^{(w)}_p=\vspan\{f_{w(1)},\ldots, f_{w(p)}\}$ for all $p$, $1\leq p\leq n$.  Note that $G=GL_K(\mathbb{C})$ acts on $\Fl_\alpha(\bC^K)$ via its action on $\bC^K$ and so does its subgroup $B\subseteq G$ of upper-triangular matrices.
Now define the {\bf Schubert cell} $C_w$ to be the $B$ orbit of $F^{(w)}_\bullet$ and the {\bf Schubert variety} $X_w=\overline{C_w}$ to be its closure.  The Schubert cells form an affine paving (in fact a cell decomposition as a CW-complex) of the partial flag variety.

There are several other descriptions of the Schubert cells and Schubert varieties that will be helpful.  Given any $V_\bullet\in C_w$, for each $p$, $1\leq p\leq K$, there exists a unique vector $v_p\in V_p\setminus V_{p-1}$ such that
\begin{equation} \label{eq:SchubCellCoords}
v_p=f_{w(p)}+\sum_{h=1}^{w(p)-1} \alpha_{h,p} f_h,
\end{equation}
where $\alpha_{h,p}=0$ if $h\in\{w(1),\ldots,w(p-1)\}$.  Note $V_p=\vspan\{v_1,\ldots,v_p\}$.  The $\alpha_{h,p}$ that are not required to be 0 can be taken as algebraically independent coordinates on $C_w$, considered as a locally closed subvariety of $\Fl_{(1^n)}(\bC^K)$.

Schubert cells and Schubert varieties can also be described in terms of intersection conditions with respect to the {\bf base flag} $F_\bullet \coloneqq F_\bullet^{(\iota)}$, where $\iota : [K]\to[K]$ is the identity map. To be precise, $F_p=\vspan\{f_1,\ldots,f_p\}$ for all $p$, $1\leq p\leq n$.  Given an injective map $w : [n]\to [K]$, we have the \textbf{Schubert cell}
\begin{align} \label{eqn:SchubCellDef}
 C_w \coloneqq \{V_\bullet \in \Fl_{(1^n)}(\bC^K) \st \dim(V_i\cap F_p) = \#\{j\leq i \st w(j)\leq p\}\}
\end{align}
and the \textbf{Schubert variety}
\begin{align}
    X_w \coloneqq \{V_\bullet \in \Fl_{(1^n)}(\bC^K) \st \dim(V_i\cap F_p) \geq \#\{j\leq i\st w(j)\leq p\}\}.
\end{align}

Since the Schubert cells $C_w$ form an affine paving, the Schubert classes $\{[X_w] \st w:[n]\to [K] \text{ injective}\}$ form a basis of $H^*(\Fl_{(1^n)}(\bC^K))$ (see Lemma \ref{lem:OddCohVanishes}).

\subsection{Chern classes} 

Given a complex vector bundle $E$ on a paracompact Hausdorff space $X$, the $i$-th Chern class of $E$ is a distinguished cohomology class $c_i(E)\in H^{2i}(X)$, with $c_0(E) = 1$ by definition. The Chern classes are invariants of the vector bundle, in the sense that if two vector bundles on $X$ are isomorphic, then their Chern classes agree.  

The sum of the Chern classes of a vector bundle $c(E) \coloneqq 1+c_1(E) + c_2(E) +\cdots$ is called the {\bf total Chern class} of $E$. It has the following useful properties.
\begin{itemize}
    \item Naturality: For any continuous map $f : X \to Y$ and any complex vector bundle $E$ on $Y$, we have $f^*(c(E)) = c(f^*(E))$, where the first $f^*$ is the map on cohomology and $f^*(E)$ is the pullback of $E$.
    \item Additivity: Given a short exact sequence of vector bundles $0\to E'\to E\to E''\to 0$ on $X$, we have
    \begin{align}
        c(E) = c(E')c(E''),
    \end{align}
    where multiplication is via the cup product on cohomology.
    \item Vanishing: If $r$ is the rank of $E$ as a complex vector bundle, then $c_i(E) = 0$ for all $i>r$.
    \item Triviality: If $E \cong \bC^r \times X$, a trivial vector bundle, then $c(E) = 1$.
\end{itemize}

In the case where $X = \Fl(n)$, for each $j$, there is the tautological vector bundle $\widetilde V_j$ whose fiber over $V_\bullet = (V_1,\dots, V_n)$ is the vector space $V_{j}$. Borel~\cite{Borel} proved that the classes $-c_1(\widetilde V_j/\widetilde V_{j-1})$ generate the cohomology ring $H^*(\Fl(n))$ as a graded algebra. Moreover, there is an isomorphism of graded algebras
\begin{align}\label{eq:BorelTheorem}
    H^*(\Fl(n)) \cong \frac{\bZ[x_1,\dots, x_n]}{\langle e_1(x_1,\dots, x_n),\dots, e_n(x_1,\dots, x_n)\rangle}
\end{align}
identifying $-c_1(\widetilde V_j/\widetilde V_{j-1})$ with $x_j$, where each variable $x_j$ is considered to have degree $2$.
The quotient ring on the right-hand side of \eqref{eq:BorelTheorem} is also known as the \textbf{coinvariant ring}.

\subsection{Affine paving}

Given a complex algebraic variety $X$, an {\bf affine paving} of $X$ is a sequence of closed subvarieties
\begin{align} \label{eq:affine-paving}
    X_0 \subseteq X_1\subseteq \cdots \subseteq X_m = X
\end{align}
of $X$ such that $X_i\setminus X_{i-1} \cong \bigsqcup_j A_{i,j}$ for some locally closed subspaces $A_{i,j}$, where for all $i,j$, $A_{i,j}\cong \bC^k$ for some $k$. The affine spaces $A_{i,j}$ are called the {\bf cells} of the affine paving.

When $X$ is compact, an affine paving gives us a convenient basis of the cohomology groups of $X$. The following two lemmata are standard (see for example~\cite{Hotta-Springer}), whereas the third is less standard but will be very useful for us.

\begin{lemma}\label{lem:OddCohVanishes}
Suppose $X$ is a compact complex algebraic variety that has an affine paving. If $X_i\setminus X_{i-1} = \bigsqcup_{i,j} A_{i,j}$ is the decomposition of $X$ into affine spaces, then we have the following isomorphisms of groups
\begin{align}
    H_{2k}(X) &\cong \bZ \cdot \big\{[\overline{A_{i,j}}]\big\}_{i,j}, \\
    \label{eq:even-cohomology}
    H^{2k}(X) &\cong H_{2k}(X)^* \cong \bZ^{\#\{(i,j) \st \dim_\bC(A_{i,j}) = k\}}, \\ \label{eq:odd-cohomology}
    H^{2k+1}(X) & = 0,
\end{align}
for all $k\geq 0$, where $[\overline{A_{i,j}}]$ is the fundamental homology class of the cell closure. If ``compact'' is replaced by ``smooth'', then the vanishing \eqref{eq:odd-cohomology} holds for compactly-supported cohomology, so $H^{2k+1}_c(X) = 0$ for all $k\geq 0$.
\end{lemma}

Under certain conditions, affine pavings can also be used to prove that the map on cohomology induced by a continuous map is injective or surjective.
\begin{lemma}\label{lem:PavingSurj}
Suppose $X$ is a smooth compact complex algebraic variety and $Y\subseteq X$ is a closed subvariety of $X$. If $Y$ and $X\setminus Y$ have affine pavings, then the map on cohomology
\begin{align}
H^*(X)\to H^*(Y)
\end{align}
induced by the inclusion $Y\subseteq X$
is surjective.
\end{lemma}
\begin{proof}
By Lemma~\ref{lem:OddCohVanishes}, all odd cohomology groups of $X$ and $Y$ and all odd cohomology groups with compact support of $X\setminus Y$ are zero. By the long exact sequence for compactly supported cohomology associated to the diagram $Y\hookrightarrow X\hookleftarrow X\setminus Y$, we have short exact sequences
\begin{align}\label{eq:SESCoh}
0 \to H^{2i}_c(X\setminus Y) \to H^{2i}(X) \to H^{2i}(Y)\to 0
\end{align}
for all $i$.
The surjectivity of the map on cohomology then follows from \eqref{eq:SESCoh}.
\end{proof}

\begin{lemma}[Relative Affine Paving Lemma]\label{lem:InjectiveCoh}
Let $f: X\to Y$ be a surjective continuous map between compact complex algebraic varieties. Suppose that $Y$ has an affine paving such that for each cell $A_{i,j}$ of $Y$, we have an isomorphism
\begin{align} \label{eqn:f-isom}
    f^{-1}(A_{i,j}) \cong A_{i,j}\times Z_{i,j}
\end{align}
for some nonempty compact complex algebraic variety $Z_{i,j}$ with an affine paving.  Furthermore, suppose that the isomorphisms \eqref{eqn:f-isom} make the following diagram commute, where $\pi_1$ is the projection onto the first factor.
\begin{equation}\begin{tikzcd}
f^{-1}(A_{i,j}) \arrow[r,"\eqref{eqn:f-isom}"]\arrow[d,"f",swap] & A_{i,j} \times Z_{i,j}\arrow[dl,"\pi_1"]\\
A_{i,j} & 
\end{tikzcd}\end{equation}
Then the map on cohomology $H^*(Y)\to H^*(X)$ is injective.
\end{lemma}

\begin{proof}
 Since $Y$ has an affine paving and $f^{-1}(A_{i,j})\cong A_{i,j} \times Z_{i,j}$, it can be seen that $X$ has an affine paving with cells $A_{i,j} \times C$, where $C$ runs over all cells of $Z_{i,j}$. 
 Therefore, $H_*(X)$ is freely generated by the fundamental classes of the cell closures, $[\overline{A_{i,j} \times C}]$.

Since $Z_{i,j}$ is compact, there is a cell of $Z_{i,j}$ consisting of a single point, $c_{i,j} = \{\text{pt}\}$, giving a cell $X_{i,j} \cong A_{i, j} \times c_{i,j}$ of $X$. Note that $\overline{X_{i,j}}$ may extend outside $f^{-1}(A_{i,j})$ and so may not be isomorphic to $\overline{A_{i,j}}$. However, since $\overline{f(\overline{C})} = \overline{f(C)}$ for all subspaces $C\subseteq X$ and continuous functions $f$, we have
\begin{equation}
\overline{f(\overline{X_{i,j}})} = \overline{f(X_{i,j})} = \overline{\pi_1(A_{i,j} \times c_{i,j})} = \overline{A_{i,j}}.
\end{equation}
Letting $f_* : H_*(X)\to H_*(Y)$ be the map on homology induced by $f$, we thus have
\begin{align}
f_*([\overline{X_{i,j}}]) = [\overline{A_{i,j}}],
\end{align}
so $f_*$ is surjective. 
By the Universal Coefficient Theorem, the map $f^*$ is the dual of $f_*$, so $f^*$ is injective.
\end{proof}

\subsection{Springer fibers}

Given a partition $\lambda$ of $n$, let $N_\lambda$ be a $n\times n$ nilpotent matrix whose Jordan block sizes are recorded by $\lambda$. We say that $N_\lambda$ has {\bf Jordan type} $\lambda$. The {\bf Springer fiber} associated to $\lambda$ is
\begin{align}
    \cB^\lambda \coloneqq \{V_\bullet \in \Fl(n) \st N_\lambda V_i\subseteq V_{i-1} \text{ for all } i\leq n\}.
\end{align}
Springer constructed an action of $S_n$ on $H^*(\cB^\lambda)$ that does not come from an action on $\cB^\lambda$. We note that in this article, the action on the cohomology ring we consider differs from the one originally constructed by Springer by tensoring with the sign representation.

A remarkable property of this action is that it gives a geometric construction of all the finite dimensional irreducible representations of $S_n$. Indeed, the dimension of $\cB^\lambda$ as a complex variety is 
\begin{align}
    n(\lambda) \coloneqq \sum_i \binom{\lambda_i'}{2},
\end{align}
and the top nonzero cohomology group of $\cB^\lambda$ as an $S_n$-module is
\begin{align}
    H^{2n(\lambda)}(\cB^\lambda;\bQ) \cong S^\lambda,
\end{align}
where $S^\lambda$ is the irreducible representation of $S_n$ usually associated to $\lambda$.
Therefore, in Lie type A there is a bijection, known as the \textbf{Springer correspondence}, between Springer fibers and the irreducible $S_n$-modules up to isomorphism. 

Hotta and Springer~\cite{Hotta-Springer} proved that the map on cohomology induced by the inclusion $\cB^\lambda \subseteq \Fl(n)$,
\begin{align}
    H^*(\Fl(n)) \to H^*(\cB^\lambda),
\end{align}
is surjective and $S_n$-equivariant. Hence, by surjectivity the cohomology ring $H^*(\cB^\lambda)$ is generated by the cohomology classes $-c_1(\widetilde V_i/\widetilde V_{i-1})$. Here, we are abusing notation and writing $\widetilde V_i$ for the restriction of this vector bundle to $\cB^\lambda$. 

There is an explicit presentation of $H^*(\cB^\lambda)$ as a quotient ring extending Borel's theorem~\cite{dCP,Tanisaki}.  Let $\la'$ denote the conjugate partition to $\la$, and let $\la_i'$ be the parts of $\la'$, where $\la_i'\coloneqq 0$ for $i>\la_1$. Given $S\subseteq \{x_1,\dots, x_n\}$, define $e_d(S)$ to be the sum of all square-free products of variables in $S$ of degree $d$. Define the following ideal and quotient ring
\begin{align}
    I_\lambda &\coloneqq \langle e_d(S) \st S\subseteq \{x_1,\dots,x_n\},\, d > |S| - \la_n' - \cdots - \la_{n-|S|+1}' \rangle,\\
    R_\lambda &\coloneqq \bZ[x_1,\dots, x_n]/I_\lambda.
\end{align}
Here and throughout the paper, we consider $R_\lambda$ to be a graded ring where each variable $x_j$ has degree $2$. It follows from work of Tanisaki~\cite{Tanisaki} that there is an isomorphism of graded rings and graded $S_n$-modules
\begin{align}
    H^*(\cB^\lambda) \cong R_\lambda
\end{align}
given by identifying the cohomology class $-c_1(\widetilde V_j/\widetilde V_{j-1})$ with the variable $x_j$.

For example, when $\lambda = (2,1)$, the ideal $I_\lambda$ is generated by $e_d(S)$ where $3\geq d>0$ and $|S|=3$, or $2\geq d>1$ and $|S| = 2$, so 
\begin{align}
I_{(2,1)} &= \big\langle x_1+x_2+x_3,\, x_1x_2+x_1x_3+x_2x_3,\, x_1x_2x_3,\, x_1x_2,\, x_1x_3,\, x_2x_3\big\rangle,
\end{align}
and $H^*(\cB^{(2,1)})\cong R_{(2,1)} = \bZ[x_1,x_2,x_3]/I_{(2,1)}$.

\subsection{Symmetric functions}

The representation theory of the group $S_n$ is closely related to the theory of symmetric functions.
A {\bf symmetric function} is a formal power series in the infinite variable set $\bx = \{x_1,x_2,\dots \}$ that is invariant under any permutation of the variables.
Given $\la$ an integer partition of $n$, which we write as $\la\vdash n$,  we denote by $\ell(\la)$ be the number of (nonzero) parts of $\la$. For each $\la \vdash n$, let $e_\la(\bx)$ and $s_\la(\bx)$ denote the \textbf{elementary symmetric function} and \textbf{Schur symmetric function} indexed by $\lambda$.  As $\lambda$ ranges over all partitions of all $n$, these form bases of the ring of symmetric functions. 

The Frobenius characteristic map, which we define next, gives a connection between symmetric functions and representations of $S_n$. Given $\la\vdash n$, let $S^\la$ be the irreducible $S_n$-module indexed by $\la$, also known as a \textbf{Specht module}. Suppose a finite-dimensional $S_n$ representation $V$ (over $\bQ$) decomposes as a direct sum of Specht modules
\begin{align}
V\cong \bigoplus_{\la\vdash n}(S^\la)^{c_\la}
\end{align}
for some nonnegative integers $c_\la$. The \textbf{Frobenius characteristic} of $V$ is then defined as the symmetric function
\begin{align}
\Frob(V) = \sum_{\la\vdash n} c_\la s_\la(\bx).
\end{align}
Given a graded $S_n$-module $V = \bigoplus_{i \geq 0} V_i$ with finite-dimensional direct summands $V_i$, the \textbf{graded Frobenius characteristic} of $V$ is
\begin{align}
\Frob(V;q) = \sum_{i\geq 0} \Frob(V_i)q^i.
\end{align}
We refer the reader to~\cite{Sagan} for more details.

\subsection{The rings \texorpdfstring{$R_{n,\lambda}$}{Rn,lambda} and \texorpdfstring{$R_{n,\lambda,s}$}{Rn,lambda,s}}

We recall the definitions and properties of the rings $R_{n,\la}$ and $R_{n,\la,s}$ introduced by the first author in \cite{GriffinOSP}. These rings simultaneously generalize the cohomology ring of a Springer fiber $H^*(\cB^\lambda)$ and the Haglund--Rhoades--Shimozono ring
\begin{align}
    R_{n,k} = \frac{\bZ[x_1,\dots, x_n]}{\langle x_1^k,\dots, x_n^k,e_n,e_{n-1},\dots, e_{n-k+1}\rangle}.
\end{align}

\begin{definition}\label{def:RnLaDef}
Fix nonnegative integers $k\leq n$, a partition $\la\vdash k$, and a positive integer $s\geq \ell(\la)$. 
 The ideals $I_{n,\la}$ and $I_{n,\la,s}$ are defined by
 \begin{align}
     I_{n,\la} &= \langle e_d(S) \st S\subseteq \{x_1,\dots, x_n\}, \, d > |S| - \la_n' - \cdots - \la_{n-|S|+1}'\rangle,\\
     I_{n,\la,s} &= I_{n,\la} +  \langle x_1^s,\dots, x_n^s\rangle.
 \end{align}
 The rings $R_{n,\la}$ and $R_{n,\la,s}$ are the corresponding quotients
 \begin{align}
 R_{n,\la} &= \bZ[x_1,\dots, x_n]/I_{n,\la},\\
 R_{n,\la,s} &= \bZ[x_1,\dots, x_n]/I_{n,\la,s}.
 \end{align}
 \end{definition}
 The rings $R_{n,\la}$ and $R_{n,\la,s}$ are evidently graded by degree and carry an action of $S_n$ by permuting the variables. In~\cite{GriffinOSP}, it is shown that $R_{n,\la}$ in general has infinitely many nonzero graded components, whereas $R_{n,\la,s}$ always has finite rank.

The following two specializations of $R_{n,\la,s}$ will be particularly important to us.
\begin{itemize}
    \item When $n=k$, $R_{n,\la,s}$ specializes to the cohomology ring of a Springer fiber. Precisely, $I_{n,\la,s} = I_{\lambda}$ for any $s\geq \ell(\la)$; thus $R_{n,\la,s} = R_\la$ in this case.
    \item When $\la = (1^k)$ and $s=k$, $R_{n,\la,s}$ specializes to the Haglund--Rhoades--Shimozono ring. Indeed, we have $I_{n,(1^k),k} = I_{n,k}$; thus $R_{n,(1^k),k} = R_{n,k}$.
\end{itemize}

\begin{example}
Let $n=4$, $\lambda = (2,1)$, and $s=2$. Then the ideal $I_{4,(2,1)}$ is generated by the polynomials 
$e_d(S)$ for $S\subseteq \{x_1,\dots, x_4\}$ such that
\begin{equation*}
\begin{aligned}
	d &= 2 \text{ and }|S| = 4,\\
	d &= 4\text{ and }|S| = 4,\\
\end{aligned}
\qquad
\begin{aligned}
	d &= 3\text{ and }|S| = 4,\\
	d &= 3\text{ and }|S| = 3.
\end{aligned}
\end{equation*}
We have
\begin{align*}
    I_{4,(2,1)} 
    &= \big\langle x_1x_2+x_1x_3+x_1x_4+x_2x_3+x_2x_4+x_3x_4,\\
    &\qquad x_1x_2x_3+x_1x_2x_4+x_1x_3x_4+x_2x_3x_4,\\ 
    &\qquad x_1x_2x_3x_4,\, x_1x_2x_3,\, x_1x_2x_4,\, x_1x_3x_4,\, x_2x_3x_4\big\rangle,
\end{align*}
and $I_{4,(2,1),2} = I_{4,(2,1)} + \langle x_1^2,x_2^2,x_3^2,x_4^2\rangle$. Finally, $R_{4,(2,1)} = \bZ[x_1,x_2,x_3,x_4]/I_{4,(2,1)}$ and $R_{4,(2,1),2} = \bZ[x_1,x_2,x_3,x_4]/I_{4,(2,1),2}$.
\end{example}

Let $I_{n,\lambda,s}^\bQ$ be the ideal in $\bQ[x_1,\dots,x_n]$ given by the same generators as $I_{n,\lambda,s}$, and let $R_{n,\la,s}^\bQ = \bQ[x_1,\dots, x_n]/I_{n,\lambda,s}^\bQ$. 
There is a convenient basis of $R_{n,\la,s}^\bQ$ generalizing the \textbf{Artin basis} of the coinvariant ring, defined next. Set $\cA_{0,\emptyset,s} \coloneqq \{1\}$. Let the set $\cA_{n,\la,s}$ be defined recursively for $n\geq 1$ by
\begin{align}
    \cA_{n,\la,s} \coloneqq \bigsqcup_{i=1}^{\ell(\la)} x_n^{i-1} \cA_{n-1,\la^{(i)},s} \sqcup \bigsqcup_{i=\ell(\la)+1}^s x_n^{i-1} \cA_{n-1,\la,s},
\end{align}
where for $1\leq i \leq \ell(\la)$, $\la^{(i)}$ is the partition obtained by sorting the parts of \[(\la_1,\dots, \la_{i-1},\la_i-1,\la_{i+1},\dots, \la_{\ell(\la)})\]
and deleting a trailing zero if necessary.
It is shown in~\cite[Theorem 3.18]{GriffinOSP} that $\cA_{n,\la,s}$ is a $\bQ$-basis of $R_{n,\la,s}^\bQ$. In fact, it is straightforward to show that $\cA_{n,\la,s}$ is actually a $\mathbb{Z}$-basis of $R_{n,\la,s}$.

\begin{lemma}\label{lem:FreeZMod}
The ring $R_{n,\la,s}$ is a free $\bZ$-module, and the set $\cA_{n,\la,s}$ is a $\bZ$-basis of $R_{n,\la,s}$.
\end{lemma}
\begin{proof}
The proof of Lemma 3.14 in \cite{GriffinOSP} also proves that $\cA_{n,\la,s}$ is a $\bZ$-spanning set of $R_{n,\la,s}$. Since $\cA_{n,\la,s}$ is a $\bQ$-linearly independent subset of $R_{n,\la,s}^\bQ$, it is also a $\bZ$-linearly independent subset of $R_{n,\la,s}$, and $R_{n,\la,s}$ is free as a $\bZ$-module.
\end{proof}

Given $V = \bigoplus_{i\geq 0} V_i$ a graded free $\bZ$-module with graded pieces $V_i$ of finite rank $\mathrm{rk}(V_i)$, let the \textbf{Hilbert--Poincar\'e series} of the module $V$ be 
\begin{align}
    \Hilb(V;q) \coloneqq \sum_{i\geq 0} \mathrm{rk}(V_i)q^i.
\end{align}
Under our convention that $x_i$ has degree $2$ for all $i$, we have the following recursive formula for the Hilbert series, which follows immediately by Lemma~\ref{lem:FreeZMod}.
\begin{lemma}\label{lem:RHilbRecursion}
We have
\begin{align}
    \Hilb(R_{n,\la,s};q) = \sum_{i=1}^{\ell(\la)} q^{2(i-1)} \Hilb(R_{n-1,\la^{(i)},s};q) + \sum_{i=\ell(\la)+1}^s q^{2(i-1)} \Hilb(R_{n-1,\la,s};q).
\end{align}
\end{lemma}

Since the set of generators of the homogeneous ideal $I_{n,\la,s}$ is closed under the action of $S_n$ permuting the variables, $R_{n,\la,s}$ inherits the structure of a graded $S_n$-module. In order to prove our generalization of the Springer correspondence, we make use of a formula for the graded Frobenius characteristic of $R_{n,\la,s}^\bQ$ proven in~\cite{GriffinOSP}. We state the formula and define the associated combinatorial objects in Section~\ref{sec:IrreducibleComponents} where we need it.

\section{An affine paving of the \texorpdfstring{$\Delta$}{Delta}-Springer variety}\label{sec:AffinePaving}

In this section, we define a family of varieties $Y_{n,\lambda,s}$ that generalize the Springer fibers, which we call $\Delta$-Springer varieties. We construct an affine paving of $Y_{n,\lambda,s}$ by intersecting it with Schubert cells.
We show that this affine paving has an inductive structure that allows us to show $H^*(Y_{n,\lambda,s})$ and $R_{n,\lambda,s}$ have the same Hilbert--Poincar\'e series.

Our approach most closely resembles that of Tymoczko~\cite{Tymoczko-LinearConditions} for type A regular Hessenberg varieties, which was extended to arbitrary type by Precup~\cite{Precup-AffinePavings}. Our work is also inspired by  Spaltenstein~\cite[Chapitre II, Proposition 5.9]{Spaltenstein-book}, who gave an affine paving for the case of Springer fibers.  Similar constructions of affine pavings have also been given by Shimomura~\cite{Shimomura} and Fresse~\cite{Fresse}.  

Let $0\leq k\leq n$ be integers, let $\la\vdash k$, and let $s$ be a positive integer such that $\ell(\lambda) \leq s$. Define $\Lambda \coloneqq \Lambda(n,\lambda,s) \coloneqq (n-k+\lambda_1,\ldots,n-k+\lambda_s)$ a partition (where $\la_i=0$ for all $i>\ell(\la)$) of $K\coloneqq |\Lambda|=s(n-k)+k$. We define the $\Delta$-Springer variety, which is our main object of study.
\begin{definition}
Let $N_\Lambda$ be a nilpotent matrix of Jordan type $\Lambda$. Define the \textbf{$\Delta$-Springer variety}
\begin{align}
    Y_{n,\la,s}\coloneqq \{V_\bullet \in \Fl_{(1^n)}(\bC^K)\st N_\Lambda V_i\subseteq V_{i-1}\text{ for }i\leq n,\text{ and }N_\Lambda^{n-k}\bC^K\subseteq V_n\},
\end{align}
where $V_0\coloneqq 0$.
\end{definition}

\begin{remark}
Since $N_\Lambda$ is nilpotent, it can be checked that the set of partial flags $V_\bullet\in \Fl_{(1^n)}(\bC^K)$ satisfying the conditions $N_\Lambda V_i\subseteq V_{i-1}$ for $i\leq n$ is the same as the set of flags satisfying the conditions $N_\Lambda V_i\subseteq V_{i}$ for $i\leq n$. Therefore, the variety $Y_{n,\la,s}$ can alternatively be defined as the reduced scheme of points satisfying
\begin{align}
    Y_{n,\la,s} = \{V_\bullet \in \Fl_{(1^n)}(\bC^K) \st N_\Lambda V_i\subseteq V_{i} \text{ for }i\leq n,\text{ and }N_\Lambda^{n-k}\bC^K\subseteq V_n\}.
\end{align}
\end{remark}

\begin{remark}[Base case $n=0$]\label{rmk:BaseCase}
For clarity in induction proofs below, we describe the case $n=0$. In this case, we must have $k=0$, $\lambda = \emptyset$, and $s>0$ arbitrary. Therefore, the ideal $I_{0,\emptyset,s}$ has no generators, so $I_{0,\emptyset,s}=\langle 0\rangle$, and since $n=0$, the ring $R_{0,\emptyset,s}$ is the quotient of the $\bZ$-algebra with no generators, namely $R_{0,\emptyset,s} = \bZ/\langle 0\rangle = \bZ$. For the $\Delta$-Springer variety, we have that $\Fl(0)$ is a single point, which represents the trivial flag, and the $\Delta$-Springer variety is also a single point. Thus, $H^*(Y_{0,\emptyset,s}) = \bZ = R_{0,\emptyset,s}$.
\end{remark}

\begin{lemma}
\label{lem:ConjInd}
  The structure of $Y_{n,\la,s}$ as an algebraic variety is independent of the choice of $N_\Lambda$.
\end{lemma}

\begin{proof}
Let $N_\Lambda$ and $N_\Lambda'$ be two nilpotent matrices of Jordan type $\Lambda$, and let $Y_{n,\la, s}$ and $Y_{n, \la,s}'$ be the corresponding $\Delta$-Springer varieties. Since $N_\Lambda$ and $N_{\Lambda}'$ have the same Jordan type, there exists $g\in \GL_K(\bC)$ such that $N_\Lambda' = gN_\Lambda g^{-1}$. It follows that $g Y_{n,\lambda, s} = Y_{n,\lambda, s}'$ under the action of $\GL_K(\bC)$ on $\Fl_{(1^n)}(\bC^K)$.  In particular, $Y_{n, \la, s} \cong Y_{n, \la, s}'$.
\end{proof}

We denote by $[\Lambda]$ the Young diagram of $\Lambda$, following the English convention, considered as the set
\begin{equation}
[\Lambda]=\{(i,j)\mid 1\leq i\leq \ell(\Lambda), 1\leq j\leq \Lambda_i\},
\end{equation}
where $(i,j)$ is the cell in the $i$-th row from the top and the $j$-th column from the left.
The cells in column $n-k+1$ and to the right form a copy of the Young diagram of $\lambda$, which we denote by $[\lambda]$. 
We say that a cell $(i,j)$ of $[\Lambda]$ is on the \textbf{right edge} if $j=\Lambda_i$.
See Figure~\ref{fig:Lambda} for an illustration of $[\Lambda]$ and $[\lambda]$, where the cells in the right edge of $[\Lambda]$ are shaded.

\begin{figure}
    \centering
    \includegraphics[scale=0.45]{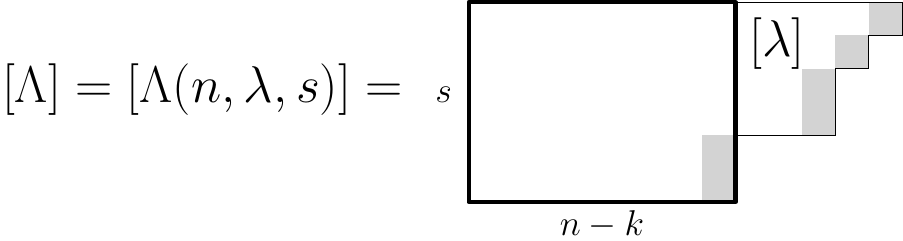}
    \caption{The Young diagram $[\Lambda]$, which has a copy of $[\lambda]$ in the upper right corner, highlighted in bold. The cells in the right edge of $[\Lambda]$ are shaded.
    }
    \label{fig:Lambda}
\end{figure}

It will be convenient to specify particular choices for $N_\Lambda$, which we do next. For any filling $T:[\Lambda]\rightarrow \mathbb{Z}_{>0}$ of $[\Lambda]$ satisfying the following conditions,
\begin{enumerate}
\item[(S1)] $T$ is a bijection between $[\Lambda]$ and $\{1,2,\dots, K\}$,
\item[(S2)] $T$ restricts to a bijection between $[\lambda]$ and $\{1, 2, \ldots, k\}$,
\end{enumerate}
we define a variety $Y_{T}$, as follows.
Let $f_1,\ldots,f_K\in \mathbb{C}^K$ be the standard ordered basis, and let $F_p=\vspan\{f_1,\ldots,f_p\}$ for all $p$ with $1\leq p\leq K$. Define $N_T$ to be the nilpotent endomorphism such that for $(i,j)\in [\Lambda]$,
\begin{align}
N_T(f_{T(i,j)}) = 
\begin{cases} 
0 & \text{ if }(i,j)\text{ is on the right edge of }[\Lambda],\\
f_{T(i,j+1)} & \text{ otherwise.}
\end{cases}
\end{align}
Note that $N_T$ has Jordan type
$\Lambda$ by construction.  
Define
\begin{align}
    Y_T \coloneqq Y_{n,\lambda,s,T}\coloneqq \{V_\bullet\in\Fl_{(1^n)}(\mathbb{C}^K) \mid N_TV_i\subseteq V_{i-1}\text{ for all } i\text{, and } F_k\subseteq V_n\},
\end{align}
which is a specific instance of the variety $Y_{n,\la,s}$.

In order to ensure that the intersection of $Y_{T}$ with the Schubert decomposition of $\Fl_{(1^n)}(\bC^K)$ is a paving by affines,
we must first put further conditions on $T$. 
We say that $T$ is {\bf $(n,\lambda,s)$-Schubert compatible} if (S1), (S2), and the following conditions hold:
\begin{enumerate}
\item[(S3)] $T$ is decreasing along each row from left to right.
\item[(S4)] For all $(i,j)\in [\lambda]$, the label $T(i,j)$ is greater than all labels in column $j+1$.
\item[(S5)] The labels of $T$ in the right edge of $[\Lambda]$ form an increasing sequence when read from top to bottom.
\item[(S6)] Whenever $T(a,b)>T(c,d)$ for $b,d>1$, then $T(a,b-1)>T(c,d-1)$.
\end{enumerate}
If $n$, $\lambda$, and $s$ are obvious from context, we will simply say $T$ is {\bf Schubert compatible}.

\begin{figure}
    \centering
    \includegraphics[scale=0.45]{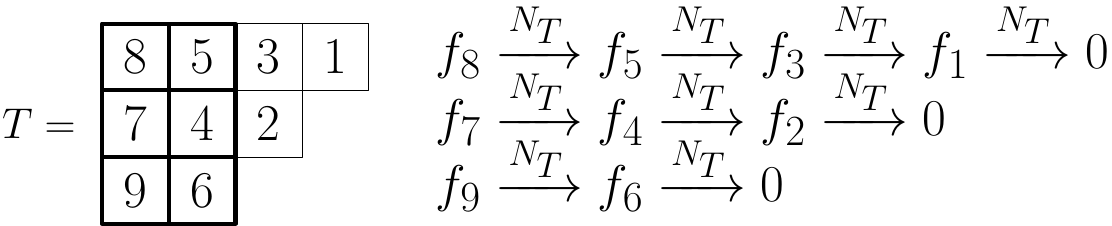}
    \caption{A Schubert-compatible filling $T$ of $\Lambda(5,(2,1),3)$ and the action of $N_T$ on the basis vectors.}
    \label{fig:SchubertCompatible}
\end{figure}

\begin{example}
Let $n=5$, $\lambda = (2,1)$, and $s=3$. Let $T$ be the Schubert-compatible filling of $\Lambda(5,(2,1),3)$ in Figure~\ref{fig:SchubertCompatible}. Then $Y_{5,(2,1),3,T}$ is the variety of partial flags $V_\bullet = (V_1,V_2,V_3,V_4,V_5)\in \Fl_{(1,1,1,1,1)}(\bC^9)$ such that the following conditions hold:
\begin{align}
N_T V_i &\subseteq V_{i-1} \text{ for } i \leq 5, \\
V_5 &\supseteq F_3 = \vspan\{f_1,f_2,f_3\}.
\end{align}
For example, the partial flag 
\begin{equation}\label{eq:PartialFlag}
\vspan\{ f_1 \} \subset \vspan\{ f_1,f_2 \} \subset 
\vspan\{ f_{1},f_{2},f_{4} \} \subset \vspan\{ f_{1},f_{2},f_{3},f_{4} \} \subset \vspan\{ f_{1},f_{2},f_{3},f_{4},f_{7}\}.
\end{equation}
is in $Y_{5,(2,1),3}$.
\end{example}

\begin{example}\label{ex:ReadingOrder}
We construct a Schubert-compatible filling $T$ for any shape $\Lambda$ as follows. 
Let the \textbf{reading order} of $[\Lambda]$ be the ordering of the cells given by scanning down the columns of $[\Lambda]$ from right to left. For $(i,j)\in [\Lambda]$, if $(i,j)$ is the $p$-th cell in the reading order, then let $T(i,j)=p$. We refer to $T$ as the filling of $[\Lambda]$ according to reading order. It can be checked that $T$ is a Schubert-compatible filling. See the left-most filling in Figure~\ref{fig:ReadingOrder} for an example of such a filling with $n=7$, $\lambda = (2,2)$, and $s=4$.
\end{example}

\begin{lemma}\label{lem:S6}
Suppose $T$ is a Schubert-compatible filling of $[\Lambda]$. If $(i,j)$ is not on the right edge of $[\Lambda]$, then 
\begin{equation}
N_T (F_{T(i,j)}\setminus F_{T(i,j)-1}) \subseteq F_{T(i,j+1)}\setminus F_{T(i,j+1)-1}.
\end{equation}
\end{lemma}

\begin{proof}
We have $N_T \, f_{T(i,j)} = f_{T(i,j+1)}$ by definition. Let $a$ and $b$ be indices such that $f_{T(a,b)}\in F_{T(i,j)}$, which means $T(a,b) < T(i,j)$. If $(a,b)$ is not on the right edge of $[\Lambda]$, by (S6) we have $T(a,b+1) < T(i,j+1)$, and hence $N_T\, f_{T(a,b)} = f_{T(a,b+1)}\in F_{T(i,j+1)}$. Otherwise, if $(a,b)$ is on the right edge, then $N_T\, f_{T(a,b)} = 0$. In either case, we have $N_T F_{T(i,j)} \subseteq F_{T(i,j+1)}$. 

If $v\in F_{T(i,j)}\setminus F_{T(i,j)-1}$, then the expansion of $v$ in the $f$ basis has a nonzero $f_{T(i,j)}$ coefficient. Therefore, the expansion of $N_T\, v$ in the $f$ basis has a nonzero $f_{T(i,j+1)}$ coefficient, so $N_T\, v\notin F_{T(i,j+1)-1}$. The lemma then follows.
\end{proof}

For $1\leq i \leq s$, define a {\bf flattening function} $\fl_T^{(i)}$ 
and a filling $T^{(i)}$ as follows.
If $i\leq \ell(\lambda)$, then $\fl_T^{(i)}$ is the unique order-preserving function with the following domain and codomain,
\begin{align}
    \fl_T^{(i)} : [K]\setminus\{ T(i,\Lambda_i) \} \to [K-1].
\end{align}
If $i> \ell(\lambda)$, then $\fl_T^{(i)}$ is the unique order preserving function
\begin{align}
    \fl_T^{(i)}: [K]\setminus(\{T(i,\Lambda_i)\}\cup \{T(i',1)\mid i'\neq i\}) \to [K-s].
\end{align}
Now if $i\leq \ell(\lambda)$, let $T^{(i)}$ be the filling obtained by deleting the last cell in row $i$, applying $\fl_T^{(i)}$ to the label in each cell, and reordering the rows so that the labels of the cells in the new right edge are increasing from top to bottom.
If $i>\ell(\lambda)$, we form $T^{(i)}$ in the same way except we also delete the cell $(i',1)$ and shift row $i'$ to the left by one unit for every $i'\neq i$ before applying $\fl_T^{(i)}$ to every label and reordering the rows.

\begin{example}
In Figure~\ref{fig:ReadingOrder} we have an example of a Schubert-compatible filling $T$ and the fillings $T^{(1)}$ and $T^{(3)}$. When constructing $T^{(3)}$, the cells labeled by $7,13,14$, and $16$ are deleted, and rows $1,2$ and $4$ are shifted left by one unit. The cells are relabeled as follows: $\fl_T^{(3)}(8)=7$, $\fl_T^{(3)}(9) = 8$, $\fl_T^{(3)}(10) = 9$, $\fl_T^{(3)}(11)=10$, $\fl_T^{(3)}(12) = 11$, and $\fl_T^{(3)}(15)=12$. Then rows $3$ and $4$ are swapped to obtain $T^{(3)}$. It can be checked that both $T^{(1)}$ and $T^{(3)}$ are Schubert compatible.
\end{example}

Recall that, given a partition $\lambda$, then $\lambda^{(i)}$ is defined to be the partition whose Young diagram is obtained from $[\lambda]$ by removing one box from the $i$-th row and then reordering the rows in decreasing order (by number of boxes) if necessary.

\begin{lemma}
The filling $T^{(i)}$ is of partition shape.  If $i \leq \ell(\lambda)$, then $T^{(i)}$ is $(n-1,\lambda^{(i)},s)$-Schubert compatible. If $i > \ell(\lambda)$, then $T^{(i)}$ is $(n-1,\lambda,s)$-Schubert compatible.
\end{lemma}

\begin{proof}
If $i\leq\ell(\la)$, the condition (S4) forces $T^{(i)}$ to have partition shape after sorting the rows by the labels in the right edge, and $T^{(i)}$ is of shape $\Lambda(n-1,\la^{(i)},s)$ since one box is removed from the $i$-th row of $\lambda$ and the rows are reordered in decreasing order.  If $i> \ell(\la)$, then $T^{(i)}$ is of shape $\Lambda(n-1,\la,s)$ since one box is removed from each row and, by (S2), any reordering only affects rows below the $\ell(\la)$-th row.  It also follows by construction that (S1) and (S2) hold for $T^{(i)}$.

The operations of deleting a cell, applying the flattening function to the labels, and possibly shifting a row to the left all preserve (S3), so $T^{(i)}$ has property (S3). Since (S4) only concerns labels of $[\lambda]$, and either all cells of $[\lambda]$ remain in place (except the one that is removed) or all cells of $[\la]$ are shifted left one column during the process of constructing $T^{(i)}$, we see $T^{(i)}$ also satisfies (S4). The property (S5) is automatically satisfied since we resort the rows by the label in the rightmost cell. Finally, $T^{(i)}$ satisfies (S6) since deleting a cell, relabeling, swapping rows, and shifting a row to the left all preserve the property (S6). Therefore, $T^{(i)}$ is Schubert compatible. 
\end{proof}

\begin{figure} 
    \centering
    \includegraphics[scale=0.45]{Figures/ReadingOrder.pdf}
    \caption{The Schubert-compatible filling $T$ of $[\Lambda] = [\Lambda(7,(2,2),4)]$ determined by reading order and the fillings $T^{(1)}$ and $T^{(3)}$, which are also Schubert compatible.}
    \label{fig:ReadingOrder}
\end{figure}

Recall that the set of injective maps $w:[n]\rightarrow [K]$ indexes the Schubert cells of $\Fl_{(1^n)}(\mathbb{C}^K)$.
\begin{definition}
Given $w:[n]\rightarrow[K]$ injective, we say that $w$ is {\bf admissible} with respect to $T$ if both of the following hold.
\begin{itemize}
    \item[(A1)] The image of the map $w$ contains $[k]$.
    \item[(A2)] For $i\leq n$, if $w(i)=T(a,b)$ for $(a,b)$ not on the right edge of $[\Lambda]$, then $T(a,b+1)\in \{w(1),\dots, w(i-1)\}$. 
\end{itemize}
\end{definition}

\begin{lemma}\label{lem:NonemptyIntersections}
Assume $T$ is a Schubert-compatible filling.  Then  $C_w\cap Y_{T}\neq\emptyset$ if and only if $w$ is admissible.
\end{lemma}

\begin{proof}
If $w$ is admissible, then the partial permutation flag $F^{(w)}_\bullet$ is in $C_w\cap Y_{T}$, so $C_w\cap Y_{T} \neq \emptyset$. Therefore, it suffices to prove that if $C_w\cap Y_{T}\neq\emptyset$, then $w$ is admissible.

Given an injective map $w: [n]\to [K]$, recall that 
\begin{align} \label{eqn:RecallSchubCellDef}
    C_w = \{V_\bullet \in \Fl_{(1^n)}(\bC^K) \st \dim(V_i\cap F_j) = \#\{p\leq i \st w(p)\leq j\}\}.
\end{align}
Given $V_\bullet\in \Fl_{(1^n)}(\bC^K)$, then $F_k\subseteq V_n$ if and only if $\dim(V_n\cap F_k) = k$. Therefore, $F_k\subseteq V_n$ for some $V_\bullet\in C_w$ if and only if (A1) holds.

Suppose $C_w\cap Y_{T}\neq\emptyset$, and let $V_\bullet \in C_w\cap Y_{T}$. Suppose there exists $i\leq n$ such that $w(i) = T(a,b)$ for some cell $(a,b)$ not on the right edge of $[\Lambda]$.
Then $\dim(V_i\cap F_{T(a,b)})>\dim(V_i\cap F_{T(a,b)-1})$, so $V_i\cap (F_{T(a,b)}\setminus F_{T(a,b)-1})\neq\emptyset$.  By Lemma~\ref{lem:S6}, we have $N_T (F_{T(a,b)}\setminus F_{T(a,b)-1}) \subseteq F_{T(a,b+1)}\setminus F_{T(a,b+1)-1}$. Hence,
\begin{align}
    N_T V_i \cap (F_{T(a,b)} \setminus F_{T(a,b+1)-1}) \neq \emptyset,
\end{align}
and since $N_T V_i\subseteq V_{i-1}$, then
\begin{align}
    V_{i-1} \cap (F_{T(a,b+1)}\setminus F_{T(a,b+1)-1}) \neq \emptyset.
\end{align}
Therefore, $\dim(V_{i-1}\cap F_{T(a,b+1)})>\dim(V_{i-1}\cap F_{T(a,b+1)-1})$, which implies by \eqref{eqn:RecallSchubCellDef} that $T(a,b+1)=w(i')$ for some $i'\leq i-1$. Hence, (A2) holds and $w$ is admissible.
\end{proof}

We define a linear transformation related to $N_T$ that we use throughout the paper.
\begin{definition}\label{def:NTranspose}
Define the nilpotent endomorphism $N^t_T$ of $\bC^K$ on the basis $\{f_1,\dots, f_K\}$ by
\begin{align}
N^t_T \,f_{T(i,j)} \coloneqq \begin{cases} f_{T(i,j-1)} & \text{ if } j>1,\\ 0 & \text{ if }j = 1.\end{cases}
\end{align}
\end{definition}
Our notation is motivated by the fact that $N^t_T$ is the transpose of $N_T$ with respect to the ordered basis $\{f_i\}$. The transformation $N^t_T$ has the crucial property that 
\begin{align}\label{eq:NNTranspose}
N_TN^t_T \,f_{T(i,j)} = \begin{cases} f_{T(i,j)} &\text{ if } j > 1,\\ 0 &\text{ if } j = 1.\end{cases}
\end{align}
In particular, $I - N_TN_T^t$ vanishes on $\mathrm{im}(N_T)$. Similarly, $I - N_T^t N_T$ vanishes on $\mathrm{im}(N_T^t)$, which is complementary to $\mathrm{ker}(N_T)$.

We now define a family of unipotent linear transformations on $\bC^K$.

\begin{definition}\label{newdef:InvertibleTransf}
Let $i$ be an integer with $1\leq i\leq s$, and let $v$ be a vector in 
$\vspan\{f_{T(h,\Lambda_h)}\st h<i\}$. 
Define the linear map $U = U_{i,v}: \bC^K\to \bC^K$ by setting
\begin{alignat}{4}
    \label{neweq:InvertibleTransfEq1}
    U f_{T(i', j)} &= f_{T(i', j)} && \text{ for all }(i',j)\text{ where } i' \ne i, \\
    \label{neweq:InvertibleTransfEq2}
    U f_{T(i, \Lambda_i)} &= f_{T(i, \Lambda_i)} + v, \\
    \label{neweq:InvertibleTransfEq3}
    U f_{T(i, j)} &= N_T^t U f_{T(i, j+1)} && \text{ recursively for all } j < \Lambda_i.
\end{alignat}
Note that \eqref{neweq:InvertibleTransfEq2} and \eqref{neweq:InvertibleTransfEq3} together are equivalent to $U f_{T(i, j)} = f_{T(i, j)} + (N_T^t)^{\Lambda_i - j} v$.
\end{definition}

\begin{example}\label{ex:6223unip}
Let $n=6$, $\lambda=(2,2)$, and $s=4$.  Now let $T$ be the filling
\[
\includegraphics[scale=0.45]{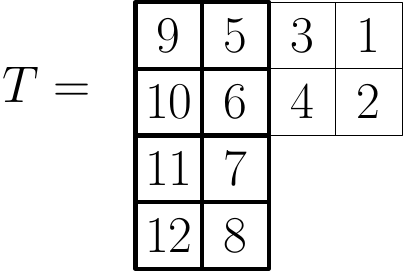}.
\]

If $i=2$ and $v=af_1$, then $U = U_{i,v}$ is given by the matrix (in terms of the basis $f_1,\ldots,f_{12}$):
\setcounter{MaxMatrixCols}{12}
\begin{align}
    U=\begin{bmatrix}
    1 & a & 0 & 0 & 0 & 0 & 0 & 0 & 0 & 0 & 0 & 0 \\
    0 & 1 & 0 & 0 & 0 & 0 & 0 & 0 & 0 & 0 & 0 & 0 \\
    0 & 0 & 1 & a & 0 & 0 & 0 & 0 & 0 & 0 & 0 & 0 \\
    0 & 0 & 0 & 1 & 0 & 0 & 0 & 0 & 0 & 0 & 0 & 0 \\
    0 & 0 & 0 & 0 & 1 & a & 0 & 0 & 0 & 0 & 0 & 0 \\
    0 & 0 & 0 & 0 & 0 & 1 & 0 & 0 & 0 & 0 & 0 & 0 \\
    0 & 0 & 0 & 0 & 0 & 0 & 1 & 0 & 0 & 0 & 0 & 0 \\
    0 & 0 & 0 & 0 & 0 & 0 & 0 & 1 & 0 & 0 & 0 & 0 \\
    0 & 0 & 0 & 0 & 0 & 0 & 0 & 0 & 1 & a & 0 & 0 \\
    0 & 0 & 0 & 0 & 0 & 0 & 0 & 0 & 0 & 1 & 0 & 0 \\
    0 & 0 & 0 & 0 & 0 & 0 & 0 & 0 & 0 & 0 & 1 & 0 \\
    0 & 0 & 0 & 0 & 0 & 0 & 0 & 0 & 0 & 0 & 0 & 1
    \end{bmatrix}.
\end{align}

\end{example}

\begin{example}\label{ex:NTransposeShift}
Let $n = 5$, $\lambda = (2, 1)$, $s = 3$, and let $T$ be as shown in Figure~\ref{fig:NTransposeShift}.
If $i = 4$ and $v = f_8+2f_3+f_1$, then $U = U_{i,v}$ takes $f_j \mapsto f_j$ for all $j \ne 10, 11$, and
\begin{alignat}{3}
    U(f_{10}) &= f_{10} + v &= f_{10} + f_8 + 2f_3 + f_1, \\
    U(f_{11}) &= f_{11} + N_T^t(v) &{}= f_{11} + f_9 + 2f_6 + f_2.
\end{alignat}
We may visualize the basis vectors that appear in $U(f_{10})$ and $U(f_{11})$ with the middle and right diagrams in Figure~\ref{fig:NTransposeShift},
where the indices of the basis vectors that appear in $U(f_{10})$ are highlighted on the left, the basis vectors of $U(f_{11})$ are highlighted on the right, and the indices of the leading terms $f_{10}$ and $f_{11}$ are circled. Note that in general the extra summand in the definition of $U(f_{T(p,q)})$ in the $p=i$ case is ``$v$ shifted left by $\Lambda^i-q$ cells'' and that this definition depends on $T$.
\end{example}

\begin{figure}[t]
    \centering
    \includegraphics[scale=0.45]{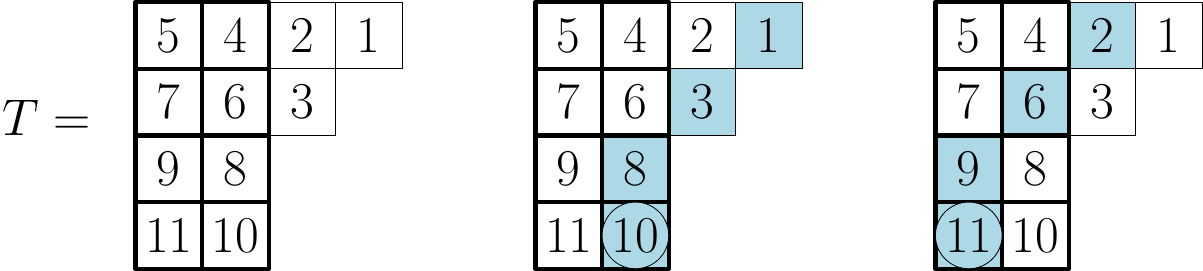}
    \caption{The Schubert-compatible filling $T$ used in Example~\ref{ex:NTransposeShift} on the left, and illustrations of the basis vectors that appear in $U(f_{10})$ and $U(f_{11})$ in the middle and right, respectively.}
    \label{fig:NTransposeShift}
\end{figure}

\begin{lemma}\label{lem:InvertibleTransf}
Let $U = U_{i,v}$ be the linear map defined in Definition \ref{newdef:InvertibleTransf}. Then $U$
is unipotent, is upper triangular with respect to the ordered basis $f_1,\ldots,f_K$, and satisfies the equation $N_TU = UN_T$.
\end{lemma}

\begin{proof}
It suffices to check the following two properties for each $p$:
\begin{enumerate}
    \item $Uf_p$ is of the form $Uf_p = f_p + \sum_{h< p} a_h f_h$, and
    \item $N_TUf_p = UN_Tf_p$.
\end{enumerate}
First, in the case $p = T(i',j)$ for $i'\neq i$, as in \eqref{neweq:InvertibleTransfEq1}, properties (1) and (2) trivially hold. Second, in the case $p=T(i,\Lambda_i)$, as in \eqref{neweq:InvertibleTransfEq2}, then (1) follows by our choice of $v$ and condition (S5) of Schubert compatibility. Furthermore, (2) holds because both sides are zero.

Finally, suppose $p = T(i,j)$ for $j<\Lambda_i$, as in \eqref{neweq:InvertibleTransfEq3}. Then $(i, j)$ is not the rightmost cell in its row, and we have by definition
\begin{equation}
Uf_{T(i, j)} = N_T^t U f_{T(i, j+1)}.
\end{equation}
By reverse induction on $j$, property (1) holds for $p=T(i,j+1)$. Property (1) for $p=T(i,j)$ then follows by condition (S6) of Schubert compatibility and the fact that $f_{T(i, j)} = N_T^t f_{T(i, j+1)}$.

It remains to show property (2) holds in this case. By the choice of $v$ and reverse induction on $j$, all the terms of the expansion of $Uf_{T(i,j+1)}$ in property (1) correspond to cells of $T$ in column $j+1$ and to the right. In particular,
\begin{equation} \label{eq:column j+1}
U f_{T(i, j+1)} \in \mathrm{im}(N_T).
\end{equation}
We compute
\begin{align}
    (UN_T - N_TU)f_{T(i, j)} &= U(N_Tf_{T(i, j)}) - N_T (U f_{T(i, j)}) \\
    &= U (f_{T(i, j+1)}) - N_T (N_T^t U f_{T(i, j+1)}) \text{ (by \eqref{neweq:InvertibleTransfEq3})}\\
    &= (I - N_TN_T^t) U f_{T(i, j+1)},
\end{align}
which is zero because $I - N_T N_T^t$ vanishes on $\mathrm{im}(N_T)$. Thus, property (2) holds.
\end{proof}

We extend the definition of $\fl_T^{(i)}$ so that it is also a function on admissible injective maps.  In particular, let $w$ be admissible with respect to $T$, and let $w(1) = T(i, \Lambda_i)$. Define 
\begin{align}
    \fl_T^{(i)}(w): [n-1] \rightarrow [K-1] & \text{ if } i\leq \ell(\lambda),\\
    \fl_T^{(i)}(w): [n-1] \rightarrow [K-s] & \text{ if } i>\ell(\lambda)
\end{align}
to be the injective maps where $\fl_T^{(i)}(w)(j)= \fl_T^{(i)}(w(j+1))$ for $1\leq j\leq n-1$.
It is not hard to check directly that if $w$ is admissible with respect to $T$, then $\fl_T^{(i)}(w)$ is admissible with respect to $T^{(i)}$. This also follows from the next lemma.

\begin{definition}
 Define linear maps
\begin{align}
    \phi^{(i)} : \bC^{K-1} \to \bC^K & \,\,\,\text{  for  }i\leq \ell(\la),\\
    \phi^{(i)} : \bC^{K-s} \to \bC^K & \,\,\,\text{  for  }i > \ell(\la),
\end{align}
by setting $\phi^{(i)}(f_j) \coloneqq f_{(\fl_T^{(i)})^{-1}(j)}$ and extending linearly.
\end{definition}

\begin{lemma}\label{lem:CellRecursion}
We have an isomorphism
  \begin{equation}
\Phi: \vspan\{f_{T(h,\Lambda_h)}\st h<i\} \times (C_{\fl_T^{(i)}(w)} \cap Y_{T^{(i)}}) \to C_w\cap Y_T,
 \end{equation}
with the flag $\Phi(v, V_\bullet)$ defined by, for $j=1, \ldots, n$,
\begin{equation}
\Phi(v, V_\bullet)_j := U_{i,v}\big(\vspan\{f_{w(1)}\} + \phi^{(i)}V_{j-1}\big).
\end{equation}
In particular, inverting $\Phi$ gives an isomorphism
 \begin{equation}
 C_w \cap Y_T \cong \bC^{i-1} \times (C_{\fl^{(i)}_T(w)}\cap Y_{T^{(i)}}).
 \end{equation}
\end{lemma}

\begin{proof}
We first show that the image of $\Phi$ is indeed contained in $C_w\cap Y_T$.
First, let $V_\bullet \in C_{\fl_T^{(i)}(w)}\cap Y_{T^{(i)}}$. We examine $\Phi(0, V_\bullet)$. Since $U_{i,0}$ is the identity map, we have for each $j$
\begin{align}
    \Phi(0,V_\bullet)_j =  \vspan\{f_{w(1)}\} + \phi^{(i)}(V_{j-1}).
\end{align}
Since $\fl^{(i)}_{T}$ is order preserving, $\dim(\phi^{(i)}(V_j)\cap F_m)=\dim(V_j\cap F_{\fl^{(i)}_{T}(m)})$ for $m$ in the domain of $\fl_T^{(i)}$, so $\Phi(0,V_\bullet) \in C_w$. Since $w$ is admissible, $F_k \subseteq \Phi(0, V_\bullet)_n$. Next, we check that $\Phi(0, V_\bullet)$ is preserved by $N_T$. For $j=1$, we have $N_Tf_{w(1)} = 0$. For $j \geq 2$, we use the fact that for any $z$ in the domain of $\phi^{(i)}$, $N_T\phi^{(i)}(z)=\alpha f_{w(1)} + \phi^{(i)}N_{T^{(i)}}(z)$ for some $\alpha\in\mathbb{C}$. We calculate:
\begin{align}
    N_T\big(\vspan\{f_{w(1)}\} + \phi^{(i)}V_{j-1}\big) &= N_T \phi^{(i)}V_{j-1} \\
    &\subseteq \vspan\{f_{w(1)}\} +  \phi^{(i)}N_{T^{(i)}}V_{j-1} \\
    &\subseteq \vspan\{f_{w(1)}\} +  \phi^{(i)}V_{j-2}.
\end{align}
Hence, $\Phi(0, V_\bullet) \in C_w \cap Y_T$.

Now let $v\in \vspan\{f_{T(h,\Lambda_h)}\st h<i\}$ be arbitrary. Observe that $\Phi(v,V_\bullet) = U_{i,v}\Phi(0,V_\bullet)$, where $U_{i,v}$ acts on the partial flag $\Phi(0,V_\bullet)$ by acting on each subspace. Since $U_{i,v}$ is upper triangular by Lemma~\ref{lem:InvertibleTransf}, it preserves the Schubert cell $C_w$, so $\Phi(v,V_\bullet)\in C_w$, and (since $w$ is admissible) $F_k \subseteq \Phi(v,V_\bullet)_n$. Finally, since $U_{i,v}$ commutes with $N_T$ by Lemma~\ref{lem:InvertibleTransf}, $N_T$ preserves $\Phi(v, V_\bullet)$:
\begin{equation}
N_T\Phi(v,V_\bullet)_j = N_TU_{i,v}\Phi(0,V_\bullet)_j = U_{i,v}N_T\Phi(0,V_\bullet)_j \subseteq  U_{i,v}\Phi(0,V_\bullet)_{j-1} = \Phi(v,V_\bullet)_{j-1}.
\end{equation}
Hence, $\Phi(v, V_\bullet) \in C_w \cap Y_T$.

In order to show that $\Phi$ is an isomorphism, we show that $\Phi$ has an inverse. Define the linear map $\psi^{(i)} \coloneqq (\phi^{(i)})^t$, the transpose with respect to the ordered basis $\{f_j\}$. Explicitly, $\psi^{(i)}$ is the linear map
\begin{align}
    \psi^{(i)} : \bC^K \to \bC^{K-1} & \,\,\,\text{  for  }i\leq \ell(\la),\\
    \psi^{(i)} : \bC^K \to \bC^{K-s} & \,\,\,\text{  for  }i>\ell(\la)
\end{align}
defined by $\psi^{(i)}(f_j) \coloneqq f_{\fl_T^{(i)}(j)}$ if $j$ is in the domain of $\fl_T^{(i)}$ and $0$ otherwise.

Given $V^\prime_\bullet\in C_w\cap Y_T$, note that $\dim(V^\prime_1\cap F_{w(1)})=1$ and $\dim(V^\prime_1\cap F_{(w(1)-1)})=0$, so 
\begin{equation}
    V^\prime_1 = \vspan\{f_{w(1)} + v\} = U_{i,v} \vspan\{f_{w(1)}\}
\end{equation}
for some vector $v = \sum_{h=1}^{i-1} \alpha_h f_{T(h,\Lambda_h)}$, where $\alpha_h\in\mathbb{C}$ are some coefficients. Now note that $f_{w(1)}\in U^{-1}_{i,v}V^\prime_j$ for all $j$ and, in the case $i>\ell(\la)$, we have $V^\prime_j\subseteq \vspan\{f_{T(h,m)}\mid m>1 \mbox{ if } h\neq i\}$.
Hence,
\begin{equation}
U^{-1}_{i,v}V^\prime_j = \vspan\{f_{w(1)}\} + \phi^{(i)}(\psi^{(i)}(V^\prime_j)),
\end{equation}
and from this equality, a routine check shows that the inverse of $\Phi$ is given by
\begin{align}
    \Phi^{-1}(V^\prime_\bullet) = (v, (\psi^{(i)}U^{-1}_{i,v}(V'_2)\subseteq\cdots\subseteq\psi^{(i)}U^{-1}_{i,v}(V'_n))).
\end{align}
Note that the entries of $U_{i,v}$ are regular functions on $Y_T \cap C_w$ (see Section \ref{subsec:Flags}). Moreover, since $U_{i,v}$ is unipotent and upper triangular, the same is true of $U_{i,v}^{-1}$. Thus $\Phi$ and $\Phi^{-1}$ are algebraic maps, so $\Phi$ is an isomorphism of algebraic varieties.
\end{proof}

\begin{remark}[Combined isomorphism $\Phi$]\label{rmk:CombinedIso}
For a fixed $i$, the map $\Phi$ above works the same way for all $w$ such that $w(1) = i$; specifically, the map $U_{i,v}$ depends on $v$ and $i$ but not (the rest of) $w$. As such, all of these $\Phi$'s can be combined into a larger isomorphism. We let
\begin{equation}
    Y_T^i := \big\{V_\bullet \in Y_T : V_1 \subseteq \vspan\{f_{T(h,\Lambda_h)}\st h <i\}\big\} \subseteq Y_T.
\end{equation}
Then we obtain a combined isomorphism
\begin{equation}
    \Phi : \vspan\{f_{T(h,\Lambda_h)}\st h <i\} \times Y_{T^{(i)}} \to Y_T^i \setminus Y_T^{i-1}.
\end{equation}

By induction, $Y_T^i \setminus Y_T^{i-1}$ is a union of cells, as in the definition of affine paving \eqref{eq:affine-paving}. We give the precise characterization of cells in this affine paving below in Theorem \ref{thm:AffinePavingY} and Corollary \ref{cor:AffinePavingFillings}.
We do not use the notation $Y_T^i$ elsewhere in the paper, though we use a restriction of the combined map $\Phi$ in Section~\ref{sec:IrreducibleComponents}. 
\end{remark}

\begin{example}\label{ex:274813cell}
Let $n$, $\lambda$, $s$, and $T$ be as in Example~\ref{ex:6223unip}.  Let $w=2\, 7 \, 4\, 8\, 1\, 3$, which is admissible with respect to $T$. Then $i=w(1) = 2$. A flag $V_\bullet\in C_w\cap Y_T$ can be represented by a matrix
\begin{align} \label{eq:CoordinateExampleEq}
    \begin{bmatrix}
    a & c & d & g & 1 & 0 \\
    1 & 0 & 0 & 0 & 0 & 0 \\
    0 & ab& a & de & 0 & 1 \\
    0 & b & 1 & 0 & 0 & 0 \\
    0 & 0 & 0 & ae & 0 & 0 \\
    0 & 0 & 0 & e & 0 & 0 \\
    0 & 1 & 0 & 0 & 0 & 0 \\
    0 & 0 & 0 & 1 & 0 & 0 \\
    0 & 0 & 0 & 0 & 0 & 0 \\
    0 & 0 & 0 & 0 & 0 & 0 \\
    0 & 0 & 0 & 0 & 0 & 0 \\
    0 & 0 & 0 & 0 & 0 & 0
    \end{bmatrix},
\end{align}
where each $V_j$ is the span of the first $j$ columns, and $a$, $b$, $c$, $d$, $e$, and $g$ are arbitrary elements of $\mathbb{C}$. Identifying $\vspan\{f_{1}\}\cong \bC^1$, the map $\Phi^{-1}$ sends this matrix to the pair $(af_1,V_\bullet')$, where the flag $V'_\bullet$ is represented by
\begin{align}
    \begin{bmatrix}
    c & d & g & 1 & 0 \\
    0 & 0 & de & 0 & 1 \\
    b & 1 & 0 & 0 & 0 \\
    0 & 0 & 0 & 0 & 0 \\
    0 & 0 & e & 0 & 0 \\
    1 & 0 & 0 & 0 & 0 \\
    0 & 0 & 1 & 0 & 0 \\
    0 & 0 & 0 & 0 & 0 \\
    0 & 0 & 0 & 0 & 0 \\
    0 & 0 & 0 & 0 & 0 \\
    0 & 0 & 0 & 0 & 0
    \end{bmatrix}.
\end{align}
\end{example}

\begin{theorem}\label{thm:AffinePavingY}
If $T$ is Schubert compatible, then the intersections $C_w\cap Y_{n,\la,s,T}$ for $w$ admissible are the cells of an affine paving of $Y_{n,\la,s,T}$.
\end{theorem}

\begin{proof}
Since the Schubert cells $C_w$ are the cells of an affine paving of $\Fl_{(1^n)}(\bC^K)$, it suffices to show that each nonempty intersection $C_w\cap Y_{n,\la,s,T}$ is isomorphic to an affine space $\bC^d$ for some $d$. By Lemma~\ref{lem:CellRecursion}, $C_w\cap Y_{n,\la,s,T}$ is nonempty if and only if $w$ is admissible. We proceed by induction on $n$ to show that each of these intersections is an affine space.

In the base case when $n=0$, we have $n = k = 0$, $\lambda = \Lambda = \emptyset$ and $s > 0$ is arbitrary. Then $\bC^K = \bC^0$ and $Y_{n, \lambda, s, T} = \Fl_\emptyset(0) = \{\mathrm{pt}\}$. Thus, the only nonempty intersection is a single point.

In the inductive case where $n\geq 1$, we have
\begin{equation}
    C_w\cap Y_{T}\cong \mathbb{C}^{i-1}\times (C_{\fl_T^{(i)}(w)}\cap Y_{T^{(i)}})
\end{equation}
by Lemma~\ref{lem:CellRecursion}.  By induction, 
$C_{\fl_T^{(i)}(w)}\cap Y_{T^{(i)}}\cong \mathbb{C}^m$ for some $m$, so $C_w\cap Y_T \cong \bC^{i+m-1}$, and our induction is complete.
\end{proof}

\begin{definition}
An \textbf{injective partial row-decreasing filling} of $[\Lambda]$ is a filling of a subset of the cells of $[\Lambda]$ with the labels $1,2,\dots, n$ (without repeating any labels) such that the filled cells are right justified in their row, the labeling decreases along each row, and each cell of $[\lambda]$ is filled.

Given $w$ admissible with respect to $T$, let $\operatorname{IPRD}_T(w)$ be the injective partial row-decreasing filling of $[\Lambda]$ such that, for $1\leq i\leq n$, if $w(i) = T(a,b)$, then the cell $(a,b)$ of $[\Lambda]$ is labeled with $w(i)$.
\end{definition}

\begin{figure}[t]
  \centering
  \includegraphics[scale=0.4]{Figures/RowDecreasingBijection.pdf}
  \caption{With $T$ in reading order as shown on the left, an example of the bijective correspondence between admissible functions and injective partial row-decreasing fillings on the right.\label{fig:FillingBijectionExample}}
\end{figure}

Note that the data of an admissible $w$ with respect to $T$ is equivalent to the data of the corresponding IPRD.  See Figure~\ref{fig:FillingBijectionExample} for an example of this correspondence. Combining this with Theorem~\ref{thm:AffinePavingY}, we have the following corollary.

\begin{corollary}\label{cor:AffinePavingFillings}
The nonempty cells in the affine paving of Theorem~\ref{thm:AffinePavingY} are in bijection with injective partial row-decreasing fillings of $[\Lambda]$.
\end{corollary}

\begin{remark}
  Recall that in the case where $\la=(1^k)$ and $s=k$, the ring $R_{n,\la,s}$ specializes to the Haglund--Rhoades--Shimozono ring, $R_{n,(1^k),k} \cong R_{n,k}$. This ring has a $\bZ$-basis indexed by \textbf{ordered set partitions}, which are partitions of the set $[n]$ into a $k$-tuple of nonempty blocks $(B_1,\dots, B_k)$. Part of the motivation that led us to define the variety $Y_{n,\la,s}$ was the following bijection between cells of $Y_{n,(1^k),k}$ and ordered set partitions. Fix a Schubert compatible $T$.  Map $w$ to the ordered set partition $(B_1,\dots, B_k)$ where block $B_i$ is defined to be the set of labels in the $i$-th row of $\operatorname{IPRD}_T(w)$. It can be checked that this map is indeed a bijection. Hence, the total rank of $H^*(Y_{n,(1^k),k})$ is equal to the total rank of $R_{n,k}$.

  For example, let $n=6$ and $k=3$, and let $T$ be the Schubert-compatible filling of $\Lambda(6,(1^3),3)$ according to reading order, so
  \[
    \includegraphics[scale=0.45]{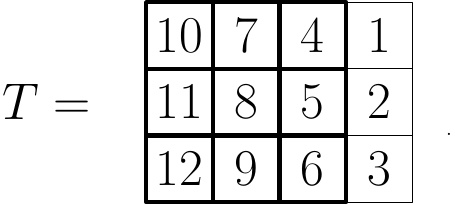}.
  \]
  Then the image of $w=253618$ under this correspondence is
  \[
    \includegraphics[scale=0.45]{Figures/OSPBijection.pdf} 
  \] 
\end{remark}

Recall our convention that $R_{n,\la,s}$  is generated in degree $2$ by the $x_i$ variables. We have the following identification of the Hilbert--Poincar\'e series of the cohomology ring of $Y_{n,\la,s}$ with that of $R_{n,\la,s}$.

\begin{theorem}\label{thm:RankGenNLaS}
The cohomology ring $H^*(Y_{n,\la,s,T})$ is a graded free $\bZ$-module concentrated in even degrees, and
\begin{align}\label{eq:HilbertEquality}
    \Hilb(H^*(Y_{n,\la,s,T});q) = \Hilb(R_{n,\la,s};q).
\end{align}
\end{theorem}
\begin{proof}
By Theorem~\ref{thm:AffinePavingY}, $Y_{n,\la,s,T}$ has a paving by affines where each cell is a copy of complex affine space. By Lemma~\ref{lem:OddCohVanishes}, all of the odd cohomology groups vanish, and $H^{2m}(Y_{n,\la,s,T})$ is a free $\bZ$-module of rank equal to the number of cells of complex dimension $m$ in the paving.

We now prove \eqref{eq:HilbertEquality} holds by induction on $n$. In the base case when $n=0$, then $\lambda=\emptyset$, $k=0$, and $T$ is the empty filling. In this case, $Y_{0,\emptyset,s,T}$ is a single point by Remark~\ref{rmk:BaseCase}, and $I_{0,\emptyset,s} = \langle 0\rangle$. Hence, $H^*(Y_{0,\emptyset,s,T})$ and $R_{0,\emptyset,s}$ are both $\bZ$, the trivial $\bZ$-algebra, and thus have the same Hilbert--Poincar\'e series.

Now suppose $n\geq 1$. By Lemma~\ref{lem:RHilbRecursion},
\begin{equation}
\Hilb(R_{n,\la,s};q) = \sum_{i=1}^{\ell(\la)} q^{2(i-1)} \Hilb(R_{n-1,\la^{(i)},s};q) + \sum_{i=\ell(\la)+1}^s q^{2(i-1)} \Hilb(R_{n-1,\la,s};q).
\end{equation}
We need to show that $\Hilb(H^*(Y_{n,\la,s});q)$ satisfies the same recursion.  

By Lemma~\ref{lem:OddCohVanishes}, whenever $T$ is a $(n,\la, s)$-Schubert compatible filling, the $q^{2m}$ coefficient of $\Hilb(H^*(Y_{n,\la,s});q)=\Hilb(H^*(Y_T);q)$ is the number of admissible $w$ such that the cell $C_w\cap Y_T$ has complex dimension $m$.  Hence,
\begin{equation}
\Hilb(H^*(Y_{n,\la,s});q) = \sum_w q^{2\dim(C_w\cap Y_T)},
\end{equation}
where the sum is over all admissible $w$.

Given a fixed $i$ with $1\leq i\leq s$, one can easily see that the map $\fl_T^{(i)}$ gives a bijection between injective maps $w$ with $w(1)=T(i,\Lambda_i)$ that are admissible with respect to $T$ and injective maps that are admissible with respect to $T^{(i)}$.  Hence, by Lemma~\ref{lem:CellRecursion}, we have
\begin{align}
\Hilb(H^*(Y_{n,\la,s});q)
&= \sum_w q^{2\dim(C_w\cap Y_T)} \\
&= \sum_{i=1}^s \sum_{w: w(1)=T(i,\Lambda_i)} q^{2\dim(C_w\cap Y_T)} \\
&= \sum_{i=1}^s \sum_{w: w(1)=T(i,\Lambda_i)} q^{2(i-1)+2\dim(C_{\fl_T^{(i)}(w)}\cap Y_{T^{(i)}})} \\
  &= \sum_{i=1}^s q^{2(i-1)} \Hilb(H^*(Y_{T^{(i)}});q).\label{eq:final-line-recursion}
\end{align}
Finally, we can split \eqref{eq:final-line-recursion} into
\[
  \sum_{i=1}^{\ell(\la)} q^{2(i-1)} \Hilb(H^*(Y_{n-1,\la^{(i)},s});q)
  + \!\sum_{i=\ell(\la)+1}^s q^{2(i-1)} \Hilb(H^*(Y_{n-1,\la,s};q)).
\]
This is the desired recursion, so the proof is complete.
\end{proof}

\begin{remark}
A detailed analysis of the recursion in Lemma \ref{lem:CellRecursion} can be used to compute the dimension of $C_w \cap Y_T$ in terms of an \emph{inversion statistic} depending on $T$ on the partial filling corresponding to $w$, when $w$ is admissible. Since we do not need this fact, we omit the proof. Below, we give a formula for the dimension of $Y_{n,\lambda,s}$ as a corollary of Theorem~\ref{thm:RankGenNLaS}.
\end{remark}

\begin{corollary}\label{cor:YDimension}
The dimension of $Y_{n,\la,s}$ is $n(\la) + (s-1)(n-k)$.
\end{corollary}

\begin{proof}
The dimension of $Y_{n,\la,s}$ is equal to the dimension of a maximal
dimensional cell in an affine paving of $Y_{n,\la,s}$, which is half
the degree of the Hilbert series $\Hilb(H^*(Y_{n,\la,s});q)$. By
Theorem~\ref{thm:RankGenNLaS}, this is exactly half the degree of $\Hilb(R_{n,\la,s};q)$, which is known to be $n(\la) + (s-1)(n-k)$; see~\cite{GriffinThesis,Rhoades-Yu-Zhao}.  (Recall again our convention that $R_{n,\la,s}$ is generated in degree $2$).
\end{proof}

\section{The iterated projective bundle \texorpdfstring{$Y_{n,\emptyset,s}$}{Yn0s}}\label{sec:EmptyPartition}

In this section, we analyze the variety $Y_{n,\lambda,s}$ in the case when $\lambda$ is the empty partition $\emptyset$. We prove that this space is an iterated projective bundle in Lemma~\ref{lem:ProjectiveBundle}. We then prove that $Y_{n,\emptyset,s}$ has the same cohomology ring as $(\bP^{s-1})^n$ in Lemma~\ref{lem:CohEmptyPartition}.  Furthermore, we show that there is a closed embedding of $Y_{n,\la,s}$ in $Y_{n,\emptyset,s}$ that induces a surjective map on cohomology, giving some of our desired relations for $H^*(Y_{n,\la,s})$.

For all $i\leq n$, let $\widetilde V_i$ be the tautological rank $i$ vector bundle on $\Fl_{(1^n)}(\bC^{ns})$ for $i\leq n$. We abuse notation and also denote by $\widetilde V_i$ the restriction of $\widetilde V_i$ to the subvariety $Y_{n,\emptyset,s}$. 

\begin{lemma}\label{lem:ProjectiveBundle}
Let $T$ be a $(n,\emptyset,s)$-Schubert-compatible filling such that the labels in the first column are $n(s-1)+1,\dots, ns$ in some order, and let $T'$ be the result of deleting the first column of $T$. Then the map
\begin{align}\label{eq:ForgettingMap}
Y_{n,\emptyset,s,T} \rightarrow Y_{n-1,\emptyset,s,T'}
\end{align}
given by forgetting the last subspace in the partial flag is a $\bP^{s-1}$-bundle map.
\end{lemma}

\begin{proof}
Given any $V_\bullet\in Y_{n,\emptyset,s,T}$, then $N_T^{n-k-1}V_{n-1} = 0$, so by our assumption on $T$ we have 
\begin{align}\label{eq:VContainedInFSpan}
V_{n-1} \subseteq \ker(N_T^{n-k-1})  = F_{n(s-1)}.
\end{align}
Furthermore, by our assumption on $T$, the nilpotent transformation $N_{T'}$ is the restriction of $N_T$ to $F_{n(s-1)}\subseteq \bC^{ns}$.  Therefore, $(V_1,\dots, V_{n-1})\in Y_{n-1,\emptyset,s,T'}$, so the map \eqref{eq:ForgettingMap} is well defined.

Given a subspace $V\subseteq \bC^{ns}$, let $N_T^{-1}(V)$ be the preimage of $V$ under the map $N_T : \bC^{ns}\to \bC^{ns}$. Observe that given $(V_1,\dots, V_{n-1})\in Y_{n-1,\emptyset,s,T'}$, an extension of this partial flag to $(V_1,\dots, V_{n-1},W)\in \Fl_{(1^n)}(\bC^K)$ is in $Y_{n,\emptyset,s}$ if and only if $W\subseteq N_T^{-1}(V_{n-1})$. We claim that for any subspace $V\subseteq F_{n(s-1)}$ of dimension $n-1$, 
\begin{align}\label{eq:DimOfInvImage}
\dim_\bC (N_T^{-1}(V)) = s+n-1.
\end{align}
Indeed, define a linear map
\begin{align}
\varphi = N_T|_{N_T^{-1}(V)} : N_T^{-1}(V)\to V,
\end{align}
which is the restriction of $N_T$. It is clear that this map is surjective, so \eqref{eq:DimOfInvImage} follows by rank-nullity and the fact that $\dim(\ker(N_T)) = s$.

Let $\widetilde{N_T^{-1}V}_{n-1}$ be the rank $s+n-1$ vector bundle on $Y_{n-1,\emptyset,s,T'}$ whose fiber over $V_\bullet$ is $N_T^{-1}(V_{n-1})$, and let $\widetilde V_{n-1}$ be the rank $n-1$ tautological vector bundle on $Y_{n-1,\emptyset,s,T'}$. We have an isomorphism
\begin{align}\label{eq:IsoOfProjBdl}
Y_{n,\emptyset,s,T} \cong \bP(\widetilde{N_T^{-1}V}_{n-1}/\widetilde V_{n-1})
\end{align}
defined by sending $V_\bullet$ to the line $V_n/V_{n-1}$ over the point $(V_1,\dots, V_{n-1})$ of $Y_{n-1,\emptyset,s,T'}$. Hence, $Y_{n,\emptyset,s,T}$ is a $\bP^{s-1}$-bundle over $Y_{n-1,\emptyset,s,T'}$ via the forgetting map~\eqref{eq:ForgettingMap}.
\end{proof}

We note that the variety $Y_{n,\emptyset,s}$ is a special case of a Steinberg variety, as defined in~\cite{Borho-MacPherson,Precup-Tymoczko-Parabolic}. Its cohomology ring is known~\cite{Borho-MacPherson} to be isomorphic to the ring of $(S_1\times \cdots\times S_1 \times S_{n(s-1)})$-invariants of the cohomology ring of the Springer fiber $H^*(\mathcal{B}^{\Lambda})$. It is not hard to prove the next lemma using this fact, but we instead give a self-contained proof for the sake of completeness.

\begin{lemma}\label{lem:CohEmptyPartition}
There is an isomorphism
\begin{align}
H^*(Y_{n,\emptyset,s})\cong \frac{\bZ[x_1,\dots, x_n]}{\langle x_1^s,\dots, x_n^s\rangle}
\end{align}
that identifies $x_i$ with $-c_1(\widetilde V_i/\widetilde V_{i-1})$
\end{lemma}

\begin{proof}
We may assume, without loss of generality, that the hypotheses in
Lemma \ref{lem:ProjectiveBundle} continue to hold. We proceed by induction on $n$. In the case $n=1$, the lemma holds since $Y_{1,\emptyset,s,T} = \bP^{s-1}$.
Suppose by way of induction that 
\begin{align}
H^*(Y_{n-1,\emptyset,s,T'}) \cong \frac{\bZ[x_1,\dots, x_{n-1}]}{\langle x_1^s,\dots, x_{n-1}^s\rangle}.
\end{align}
Let us denote by $E$ the line bundle $\widetilde{N^{-1}V}_{n-1}/\widetilde V_{n-1}$.
By~\eqref{eq:IsoOfProjBdl}, we have an isomorphism
\begin{align}
Y_{n,\emptyset,s,T} \cong \bP(E),
\end{align}
so that $\widetilde V_n/\widetilde V_{n-1} \cong \mathcal{O}_E(1)$.
Hence, by Grothendieck's construction of Chern classes, we have
\begin{align}
H^*(Y_{n,\emptyset,s,T}) \cong \frac{H^*(Y_{n-1,\emptyset,s,T'})[x_n]}{\langle x_n^s + c_1(E)x_n^{s-1} + \cdots + c_s(E)\rangle}.
\end{align}

It suffices to prove $c(E) = 1$.
Indeed, observe that if $V_\bullet\in Y_{n-1,\emptyset,s,T'}$, then $V_{n-1}\subseteq F_{n(s-1)} = \im(N_T)$. Let $\bC^{ns}$ and $\im(N_T)$ be the corresponding trivial vector bundles on $Y_{n-1,\emptyset,s,T'}$. Consider the following short exact sequence of vector bundles,
\begin{align}
0 \to E\to \bC^{ns}/\widetilde V_{n-1}\to \im(N_T)/\widetilde V_{n-1} \to 0,
\end{align}
where the second map is the composition $E\hookrightarrow \bC^{ns} \twoheadrightarrow \bC^{ns}/\widetilde V_{n-1}$ and the third map is induced by $N_T$.
Then we have the following identity of Chern classes,
\begin{align}
c(E) = \frac{c(\bC^{ns}/\widetilde V_{n-1})}{c(\im(N_T)/\widetilde V_{n-1})} = c(\bC^{ns}/\im(N_T)) = 1,
\end{align}
since $\bC^{ns}/\im(N_T)$ is a trivial bundle, which completes the proof.
\end{proof}

\begin{example}[Hirzebruch surface]\label{ex:Hirzebruch}
Consider $Y_{2, \emptyset, 2, T}$ where $T$ is the Schubert-compatible filling of $\Lambda(2, \emptyset, 2) = (2, 2)$ according to reading order, so $\ker(N_T) = \langle f_1, f_2 \rangle$. We have
\begin{equation}
Y_{2, \emptyset, 2} = \{V_\bullet \in \Fl_{(1,1)}(\bC^4) : V_1 \subset \ker(N_T), \ N_T V_2 \subset V_1\}.
\end{equation}
The same variety appears in work of Cautis and Kamnitzer~\cite{Cautis-Kamnitzer} and of Russell~\cite{Russell} on connections between Springer fibers and knot homology.

By Lemma \ref{lem:ProjectiveBundle}, forgetting $V_2$ realizes $Y_{2,\emptyset,2}$ as a $\mathbb{P}^1$-bundle over $Y_{1, \emptyset, 2} = \bP(\ker N_T) = \bP^1$ isomorphic to $\bP(E) = \bP(\widetilde{N_T^{-1}V_1} / \widetilde{V_1})$, the projectivization of a $2$-dimensional vector bundle. In fact,
\begin{equation} \label{eq:2ndHirzebruchIsom}
    Y_{2, \emptyset, 2} \cong \bP(\mathcal{O}(1) \oplus \mathcal{O}(-1)),
\end{equation}
the second Hirzebruch surface. In particular, $Y_{2,\emptyset,2}$ is not isomorphic as an algebraic variety to $\bP^1 \times \bP^1$, although $\bP^1 \times \bP^1$ has an isomorphic cohomology ring.

To see \eqref{eq:2ndHirzebruchIsom}, observe that there is an isomorphism
\begin{align}
    E = \widetilde{N_T^{-1}V_1} / \widetilde{V_1} \to (\ker(N_T)/\widetilde{V_1}) \oplus \widetilde{V_1},
\end{align}
induced by the linear map on vector spaces sending the coset $\overline{v} = v + V_1\in N_T^{-1}V_1/V_1$ to the pair $(N_TN_t^T \,\overline{v},\,N_T\,v)$. We leave it to the reader to check that this map is a well-defined isomorphism of complex vector bundles. The proof of~\eqref{eq:2ndHirzebruchIsom} is completed by observing that we have isomorphisms $\ker(N_T)/\widetilde{V_1}\cong \mathcal{O}(1)$ and $\widetilde V_1 \cong \mathcal{O}(-1)$ of vector bundles over $\bP^1 =\bP(\ker N_T)$.

We will see in Remark~\ref{rmk:s=2} that this Hirzebruch surface can also be identified as one of the components of the Springer fiber for the partition $(2,2)$.
\end{example}

Next we show that there is an embedding $Y_{n,\la,s}\rightarrow Y_{n,\emptyset,s}$ for any partition $\la$, and the induced map on cohomology is a surjection.

\begin{lemma}\label{lem:Embedding}
Let $T$ be a $(n,\la,s)$-Schubert-compatible filling of $\Lambda(n,\la,s)$, and let $T'$ be a $(n,\emptyset,s)$-Schubert-compatible filling of $\Lambda(n,\emptyset,s) = (n^s)$ such that every entry of the $i$-th row of $T$ is in the $i$-th row of $T'$. The linear map $\iota : \bC^K \hookrightarrow \bC^{ns}$ defined as the inclusion of the first $K$ coordinates in $\bC^{ns}$ induces a commutative diagram
\begin{equation}
\begin{tikzcd}
Y_{n,\lambda,s,T}\arrow[r,hookrightarrow,"\iota"]\arrow[d,hookrightarrow] & \arrow[d,hookrightarrow] Y_{n,\emptyset,s,T'}\\
\Fl_{(1^n)}(\bC^K) \arrow[r,hookrightarrow,"\iota"] & \Fl_{(1^n)}(\bC^{ns}),
\end{tikzcd}
\end{equation}
where each inclusion is a closed embedding.
\end{lemma}

\begin{proof}
Since the entries of $T$ in row $i$ are right justified in row $i$ of $T'$ for all $i$, we have $\iota \circ N_T = N_{T'} \circ \iota$. Furthermore, the containment $(N_{T'})^n \,\bC^{ns} \subseteq \iota(V_n)$ is trivial since $(N_{T'})^n = 0$. Hence, the injection $Y_{n,\la,s,T}\xrightarrow{\iota} Y_{n,\emptyset,s,T'}$ induced by $\iota$ is well defined. The commutativity of the diagram is immediate.
\end{proof}

\begin{lemma}\label{lem:AffinePavingDiff}
With the same hypotheses as Lemma~\ref{lem:Embedding}, the open subspace ${Y_{n,\emptyset,s,T'}\setminus \iota(Y_{n,\la,s,T})}$ has an affine paving.
\end{lemma}

\begin{proof}
By Theorem~\ref{thm:AffinePavingY}, the intersections $C_w\cap Y_{n,\la,s,T}$ for $w$ admissible with respect to $T$ are the cells of an affine paving of $Y_{n,\la,s,T}$, and the intersections $C_v\cap Y_{n,\emptyset,s,T'}$ for $v$ admissible with respect to $T'$ are the cells of an affine paving of $Y_{n,\emptyset,s,T'}$.

Given such a cell $C_w\cap Y_{n,\la,s,T}$, where $w$ is admissible with respect to $T$, define $w' : [n]\to [ns]$ by extending the codomain of $w$ to $[ns]$. Then $w'$ is admissible with respect to $T'$, and it can be checked that $\iota(C_w\cap Y_{n,\la,s,T}) = C_{w'}\cap Y_{n,\emptyset,s,T'}$. Therefore, $Y_{n,\emptyset,s,T'}\setminus \iota(Y_{n,\la,s,T})$ has an affine paving given by removing the cells $C_{w'}\cap Y_{n,\emptyset,s,T'}$ coming from $Y_{n,\la,s,T}$. These are the $w'$ with $[k]\subseteq \im(w')\subseteq[K]$. \end{proof}

\begin{theorem}\label{thm:Surj}
With the same hypotheses as Lemma~\ref{lem:Embedding}, the closed embedding $\iota$ induces a surjection
\begin{align}
\frac{\bZ[x_1,\dots,x_n]}{\langle x_1^s,\dots,x_n^s\rangle} \cong H^*(Y_{n,\emptyset,s,T}) \twoheadrightarrow H^*(Y_{n,\la,s,T'}).
\end{align}
\end{theorem}
\begin{proof}
This follows immediately by Lemma~\ref{lem:PavingSurj}, Theorem~\ref{thm:AffinePavingY}, and Lemma~\ref{lem:AffinePavingDiff}.
\end{proof}

We end this section by remarking on the structure of $Y_{n,\emptyset,s}$ as a real manifold. In particular, Lemma \ref{lem:CohEmptyPartition} can also be shown by topological means.

\begin{proposition} \label{prop:diffeomorphism}
There is a diffeomorphism $Y_{n, \emptyset, s} \cong (\bP^{s-1})^n$ as real manifolds. In particular, $H^*(Y_{n, \emptyset, s}) \cong H^*((\bP^{s-1})^n)$.
\end{proposition}

\begin{proof}
The argument is essentially identical to \cite[Theorem 2.1]{Cautis-Kamnitzer} (which in our notation covers the case $s=2$). Since we do not need the precise statement, we omit the details.
\end{proof}

We note however that despite Proposition \ref{prop:diffeomorphism}, $(\bP^{s-1})^n$ and the iterated projective bundle $Y_{n, \emptyset, s}$ are not in general isomorphic as varieties or complex manifolds.

\begin{remark}
As stated in the introduction, a special case of our results states that $Y_{n,(1^k),k}$ has the same cohomology ring as the Pawlowski--Rhoades spanning line arrangement space $X_{n,k}$. The latter~\cite{Pawlowski-Rhoades} is an open subvariety of $(\bP^{k-1})^n$, while the $\Delta$-Springer variety $Y_{n,(1^k),k}$ is a closed subvariety of $Y_{n,\emptyset,k}$.

In fact, $X_{n,k}$ and $Y_{n, (1^k), k}$ are further related by the fact that they have affine pavings satisfying the same recursion. In the case of spanning line arrangements, the recursion involves intermediary spaces $X_{n,k,r}$ for $0\leq r\leq k$, where the $r=0$ case is the product of projective spaces and the $r=k$ case is $X_{n,k}$. In the case of $\Delta$-Springer varieties, the corresponding intermediary spaces are $Y_{n,(1^r),k}$ by Lemma~\ref{lem:CellRecursion}, and the $r=0$ case is $Y_{n, \emptyset, k}$. Furthermore, combining 
Theorem~\ref{thm:MainThmIntro} and~\cite[Theorem 8.4]{Pawlowski-Rhoades} gives isomorphisms
\begin{equation}
H^*(X_{n,k,r}) \cong H^*(Y_{n,(1^r),k}).
\end{equation}
\end{remark}

\section{Spaltenstein varieties and the cohomology of \texorpdfstring{$Y_{n,\lambda,s}$}{Yn,lambda,s}}\label{sec:SpaltensteinAndCohomology}

In this section, we prove that there is a cellular surjective map from a Spaltenstein variety to $Y_{n,\la,s}$. We use this fact together with work of Brundan and Ostrik on the cohomology ring of a Spaltenstein variety~\cite{Brundan-Ostrik} to prove that the cohomology ring of $Y_{n,\la,s}$ is isomorphic to $R_{n,\la,s}$, stated as Theorem~\ref{thm:MainTheorem}. 

Let us outline our strategy. First, by Theorem~\ref{thm:Surj} we know that $H^*(Y_{n,\la,s})$ is a quotient of the ring
\begin{align}
\frac{\bZ[x_1,\dots,x_n]}{\langle x_1^s,\dots, x_n^s\rangle}.
\end{align}
Next, by Theorem~\ref{thm:RankGenNLaS}, the rings $R_{n,\la,s}$ and $H^*(Y_{n,\la,s})$ are free $\bZ$-modules with the same Hilbert series. Since the defining ideal of $R_{n,\la,s}$ is $I_{n,\la,s} = I_{n,\la} + \langle x_1^s,\dots,x_n^s\rangle$, it thus suffices to prove that for each generator $e_d(S)$ of $I_{n,\la}$ where $S\subseteq \{x_1,\dots, x_n\}$, the same polynomial in the first Chern classes $x_i = -c_1(\widetilde V_i/\widetilde V_{i-1})$ vanishes in $H^*(Y_{n,\la,s})$. To do this, we exhibit an injection from $H^*(Y_{n,\la,s})$ into the cohomology of a Spaltenstein variety, and we prove that the $e_d(S)$ polynomials in the first Chern classes vanish in the cohomology ring of the Spaltenstein variety using results of Brundan and Ostrik~\cite{Brundan-Ostrik}.

Let us recall the definition of a Spaltenstein variety. Given an $m\times m$ nilpotent matrix $N_\nu$ of Jordan type $\nu\vdash m$ and a composition $\mu\vDash m$ of length $\ell$, the \textbf{Spaltenstein variety} is
\begin{align}
\cB_\mu^{\nu} \coloneqq \{V_\bullet \in \Fl_{\mu_1,\,\mu_2,\dots,\, \mu_\ell}(\bC^m)\st N_\nu V_i \subseteq V_{i-1} \text{ for }i\leq \ell\}.
\end{align}
Let $X_j = \{x_{\mu_1+\cdots + \mu_{j-1}+1},\dots, x_{\mu_1+\cdots +\mu_j}\}$. Given $1\leq i_1<\cdots < i_p\leq \ell$, let 
\begin{align}
e_d(X;i_1,\dots, i_p) \coloneqq e_d(X_{i_1}\cup X_{i_2}\cup\cdots \cup X_{i_p}).
\end{align}
Furthermore, let $I_\mu^\nu$ be the following ideal of $\bZ[x_1,\dots, x_m]$,
\begin{equation}
I_\mu^\nu \coloneqq \langle e_d(X;i_1,\dots, i_p) \st 1\leq i_1 < \cdots < i_p\leq \ell\text{ and }  d > \mu_{i_1} + \cdots + \mu_{i_p} - \nu'_{\ell-p+1} -\cdots - \nu'_m \rangle.
\end{equation}
Brundan and Ostrik~\cite{Brundan-Ostrik} proved the following isomorphism of graded rings,
\begin{align}
H^*(\cB_\mu^\nu) \cong \frac{\bZ[x_1,\dots, x_m]^{S_{\mu}}}{I_{\mu}^\nu},
\end{align}
where $x_i$ on the right-hand side is identified with $-c_1(\widetilde{V}_i/\widetilde{V}_{i-1})$.

Let us take the special case where $m = K$, $\nu = \Lambda$, and $\mu = (1^n,s-1,s-1,\dots, s-1)$, where $s-1$ is repeated $n-k$ many times, so that $\ell = 2n-k$. Observe that $\Lambda'_{n-k+i} = \la'_i$ for $i\geq 0$. Further observe that for each $j\leq n$, then $X_j = \{x_j\}$. Setting $S = \{i_1,\dots, i_p\}$ for $1<i_1<\cdots <i_p \leq n$, then $e_d(S)\in I_\mu^\Lambda$ for
\begin{align}
d > p-\Lambda'_{(2n-k)-p+1} - \dots -\Lambda'_K,
\end{align}
or equivalently
\begin{align}
d > p-\la'_{n-p+1} - \dots - \la'_n.
\end{align}
The next lemma follows immediately from these observations.
\begin{lemma}\label{lem:IdealContainment}
Recall the ideal $I_{n,\la}\subseteq \bZ[x_1,\dots,x_n]$ defined in Definition~\ref{def:RnLaDef}. With $\mu$ as above, we have the containment of sets $I_{n,\la} \subseteq I_\mu^\Lambda$.
\end{lemma}

\begin{remark}
It will follow from Lemma~\ref{lem:InjCohomology} below that in fact we have the stronger containment $I_{n,\la,s}\subseteq I_\mu^\Lambda$. This is not immediate from Brundan and Ostrik's presentation of $H^*(\cB_\mu^\Lambda)$ but could likely be shown with an algebraic calculation using symmetric polynomial identities.
\end{remark}

Observe that there is a map
\begin{align}
\pi : \cB_\mu^\Lambda \rightarrow Y_{n,\la,s}
\end{align}
given by projecting onto the first $n$ parts of the partial flag. Indeed, if $V_\bullet \in \cB_\mu^\Lambda$, then $V_{2n-k} = \bC^K$ by definition. Since $N_\Lambda V_i\subseteq V_{i-1}$ for all $i$, then $\im(N_\Lambda^{n-k}) = N_\Lambda^{n-k}V_{2n-k} \subseteq V_n$, so $\pi(V_\bullet) \in Y_{n,\la,s}$. 
In order to show that the map $\pi$ is a surjective cellular map, we need the following two lemmata, the second of which is a strengthening of Lemma~\ref{lem:InvertibleTransf} that only holds for a subclass of Schubert-compatible fillings.

\begin{lemma}\label{lem:LeadingTerm}
Let $T$ be a Schubert-compatible filling. If $j>1$, then
\begin{align}
    N^t_T(F_{T(i,j)}\setminus F_{T(i,j)-1})\subseteq F_{T(i,j-1)}\setminus F_{T(i,j-1)-1}.
\end{align} 
\end{lemma}
\begin{proof}[Sketch]
The proof is an application of (S6), similar to the proof of Lemma~\ref{lem:S6}.
\end{proof}

\begin{definition}
Given a Schubert-compatible filling $T$, we say that it is \textbf{strongly Schubert compatible} if for all $(i,j)\in [\Lambda]$, the label $T(i,j)$ is greater than all labels in column $j+1$.
\end{definition}
Note that this is a strengthening of property (S4) of Schubert compatibility.

\begin{definition}
Let $T$ be a strongly Schubert-compatible filling, $w$ admissible with respect to $T$, and  $V_\bullet\in Y_{n,\la,s,T}\cap C_w$. For each $p\leq n$, let $v_p$ be the vector in $V_p\setminus V_{p-1}$ with leading term $f_{w(p)}$, as in~\eqref{eq:SchubCellCoords}.

We define a linear map $U = U(V_\bullet) : \bC^K \to \bC^K$ by setting
\begin{alignat}{4}
    \label{neweq:UnipotentEq1}
    U f_{w(p)} &= v_p && \text{ for all } 1 \leq p \leq n, \\
    \label{neweq:UnipotentEq2}
    U f_{T(i, j)} &= f_{T(i, j)} && \text{ if $\operatorname{IPRD}_T(w)$ is empty in row $i$}, \\
    \label{neweq:UnipotentEq3}
    U f_{T(i, j)} &= N_T^t U f_{T(i, j+1)} && \text{ recursively for all other cells $(i, j)$ of $[\Lambda]$.}
\end{alignat}
Note that \eqref{neweq:UnipotentEq3} does not cause $U$ to be not well defined because the filled cells in $\operatorname{IPRD}_T(w)$ are right justified.
\end{definition}

Given an injective map $w:[n]\rightarrow[K]$, recall that $F^{(w)}_\bullet$ is the partial flag $F^{(w)}_1\subseteq \cdots \subseteq F^{(w)}_n$ where $F^{(w)}_j=\vspan\{f_{w(1)},\ldots,f_{w(j)}\}$.

\begin{lemma}\label{newlem:TechnicalLemmaUnipotent}
The map $U = U(V_\bullet)$ defined above is unipotent, is upper triangular, and satisfies
\begin{align}
    \label{neweq:UnipotentEq4}
    U F^{(w)}_p &= V_p \text{ for all } 1 \leq p \leq n, \text{ and } \\
    \label{neweq:UnipotentEq5}
    (UN_T - N_TU)x &\in V_n \text{ for all } x \in \bC^K.
\end{align}
\end{lemma}

\begin{proof}
Identity \eqref{neweq:UnipotentEq4} is immediate from \eqref{neweq:UnipotentEq1} and our choice of $v_p$. We check the remaining conditions as in the proof of Lemma \ref{lem:InvertibleTransf}, examining each value $U f_p$ separately. The reasoning for unipotency and upper triangularity is identical and omitted, and we focus on identity \eqref{neweq:UnipotentEq5}.

It suffices to verify~\eqref{neweq:UnipotentEq5} for $x=f_{T(i,j)}$. We consider the three cases in the definition of $U$. In the first case $x=f_{w(p)}$, by \eqref{neweq:UnipotentEq1} we have
\begin{equation}
(UN_T - N_TU)f_{w(p)} = UN_T f_{w(p)} - N_T v_p.
\end{equation}
The first term, $UN_T f_{w(p)}$, is either $0$ or $v_q$ for some $1 \leq q < p$, since the filled cells of $\operatorname{IPRD}_T(w)$ form a rightwards-closed subset of $[\Lambda]$. The second term is $N_T v_p\in V_{p-1}$. Therefore,~\eqref{neweq:UnipotentEq5} holds for this case. It also trivially holds in the second case \eqref{neweq:UnipotentEq2}.

Finally, if $i$ and $j$ are such that $\operatorname{IPRD}_T(w)$ is nonempty in row $i$, then $Uf_{T(i, j)}$ is defined by~\eqref{neweq:UnipotentEq3}. A similar argument as in the proof of Lemma \ref{lem:InvertibleTransf} gives
\begin{equation}
(UN_T - N_TU)f_{T(i, j)} = (I - N_TN_T^t) U f_{T(i, j+1)}.
\end{equation}
Here again we have $U f_{T(i, j+1)} \in \mathrm{im}(N_T),$
so this is \emph{zero}, not just an element of $V_n$. This is because $T$ is \emph{strongly} Schubert compatible, so upper-triangularity implies that all terms of $U f_{T(i, j+1)}$ are in column $j+1$ or further to the right (compare to Equation \eqref{eq:column j+1}).
\end{proof}

\begin{lemma}\label{lem:FiberBundle}
Fix a strongly Schubert-compatible filling $T$ and an injective map $w: [n]\rightarrow [K]$ that is admissible with respect to $T$. Then we have an isomorphism
\begin{align}\label{neweq:FiberBundleIso}
\Sigma: \pi^{-1}(C_{w}\cap Y_{n,\la,s,T}) \rightarrow (C_{w}\cap Y_{n,\la,s,T})\times \cB_{(s-1)^{n-k}}^{\overline{\Lambda}},
\end{align}
where $(s-1)^{n-k} = (s-1,\dots, s-1)$ with $n-k$ parts and $\overline{\Lambda}$ is the partition obtained by first deleting cells labeled $w(1),\dots, w(n)$ from $T$, then recording the row sizes of the remaining cells in weakly decreasing order.

Furthermore, on $\pi^{-1}(C_w\cap Y_{n,\la,s,T})$, we have the equality of maps $\pi=\pi_1\circ\Sigma$, where
\begin{equation}
    \pi_1:(C_{w}\cap Y_{n,\la,s,T})\times \cB_{(s-1)^{n-k}}^{\overline{\Lambda}} \rightarrow C_w\cap Y_{n,\la,s,T}
\end{equation}
is the projection onto the first factor.
\end{lemma}

\begin{proof}
Since $N_TF^{(w)}_n\subseteq F^{(w)}_{n-1}\subseteq F^{(w)}_n$, $N_T$ induces a well-defined nilpotent endomorphism on the quotient space $\bC^K/F^{(w)}_n$.  Denote this induced endomorphism by $N_T^\prime$.  Note that $N_T^\prime$ has Jordan type $\overline{\Lambda}$.
We take as our specific instance of the Spaltenstein variety $\cB_{(s-1)^{n-k}}^{\overline{\Lambda}}$ the space of partial flags
$(W_1\subseteq\cdots\subseteq W_{n-k})\in \Fl_{(s-1)^{n-k}}(\bC^K/F^{(w)}_n)$ satisfying the condition $N'_TW_j\subseteq W_{j-1}$ for $j\leq n-k$.

We now define
\begin{equation}
\Sigma: \pi^{-1}(C_w\cap Y_T)
\rightarrow (C_w\cap Y_T) \times \cB_{(s-1)^{n-k}}^{\overline{\Lambda}}.
\end{equation}

Let $V_\bullet=(V_1\subseteq\cdots\subseteq V_n\subseteq V_{n+1}\subseteq\cdots\subseteq V_{2n-k})\in\pi^{-1}(C_w\cap Y_T)$. By Lemma~\ref{newlem:TechnicalLemmaUnipotent}, there is a unipotent upper triangular matrix $U = U(\pi(V_\bullet))$, whose entries are regular functions on $C_w \cap Y_T$, such that
\begin{align}
U F^{(w)}_p &= V_p \text{ for all } 1 \leq p \leq n, \text{ and } \\
(UN_T-N_TU)x &\in V_n \text{ for all } x \in \bC^K. \label{eq:Commutativity}
\end{align}
Define $\Sigma(V_\bullet)$ by
\begin{align}
    \Sigma(V_\bullet) = \left( (V_1\subseteq\cdots\subseteq V_n) , 
    (W_1\subseteq\cdots\subseteq W_{n-k}) \right),
\end{align}
where $W_j\coloneqq U^{-1}V_{n+j}/F^{(w)}_n$ for all $j$. Since $U$ is unipotent, the entries of $U^{-1}$ are also given by regular functions, so $\Sigma$ is regular.

We need to show that $W_\bullet=(W_1\subseteq\cdots\subseteq W_{n-k})\in \cB_{(s-1)^{n-k}}^{\overline{\Lambda}}$.  Since we have $\dim(V_{n+j})=n+(s-1)j$ and $U^{-1}V_{n+j}\supseteq U^{-1}V_n = F^{(w)}_n$, it follows that $\dim(W_j)=(s-1)j$ and so 
\begin{equation}
W_\bullet \in \Fl_{(s-1)^{n-k}}(\bC^K/F^{(w)}_n). 
\end{equation}
It remains to show $N'_T W_j \subseteq W_{j-1}$ for all $1 \leq j \leq n-k.$
We first decompose $N_T$ as the sum
\begin{equation}
    N_T = U^{-1}N_TU + U^{-1}(UN_T - N_TU).
\end{equation}
By the definition of $\cB_\mu^\Lambda$, $N_T(V_{n+j}) \subseteq V_{n+j-1}$, so \begin{equation}
U^{-1}N_TU(U^{-1} V_{n+j}) \subseteq U^{-1}N_TU(U^{-1} V_{n+j-1}).
\end{equation}
On the other hand, by~\eqref{eq:Commutativity} the second term has image contained in $U^{-1}V_n \subseteq U^{-1}V_{n+j-1}$. Combining, we see $N_T(U^{-1} V_{n+j}) \subseteq U^{-1} V_{n+j-1}$. This descends to the desired containment $N'_TW_j\subseteq W_{j-1}$.

To show $\Sigma$ is an isomorphism, we similarly define 
\begin{equation}
\Pi: (C_w \cap Y_T) \times \cB_{(s-1)^{n-k}}^{\overline{\Lambda}} \rightarrow \pi^{-1}(C_w\cap Y_T)
\end{equation}
by
\begin{align}
    \Pi(V'_\bullet,W'_\bullet)=\left(V'_1\subseteq\cdots\subseteq V'_n\subseteq U(W'_1+F^{(w)}_n)\subseteq\cdots\subseteq U(W'_{n-k}+F^{(w)}_n) \right),
\end{align}
where $U$ is obtained from $V'_\bullet$ according to Lemma~\ref{newlem:TechnicalLemmaUnipotent}.
A similar calculation to the one above verifies that the image of $\Pi$ is contained in $\pi^{-1}(C_w\cap Y_T)$.  Furthermore, $\Pi$ is regular since the map by which we obtain $U$ from $V'_\bullet$ is regular.  It can be checked that $\Pi$ and $\Sigma$ are mutual inverses, so $\Sigma$ is an isomorphism.

The last statement of the lemma follows trivially from the definition of $\Sigma$.
\end{proof}

\begin{lemma}\label{lem:PiIsSurj}
 The map $\pi$ is surjective. 
\end{lemma}

\begin{proof}
Let $T$, $w$, and $\overline{\Lambda}$ be as in Lemma~\ref{lem:FiberBundle}. By~\eqref{neweq:FiberBundleIso}, it suffices to show that $\cB_{(s-1)^{n-k}}^{\overline{\Lambda}}\neq \emptyset$. Indeed, Spaltenstein~\cite{Spaltenstein-book} proved that for any partitions $\nu,\mu$ of the same size, $\cB_\mu^\nu$ is nonempty if and only if $\mu\leq \nu'$, where $\leq$ is the dominance order on partitions defined by declaring that $\mu\leq\nu'$ if $\mu_1+\cdots +\mu_i \leq \nu'_1+\cdots + \nu'_i$ for all $i$. It is a standard result that conjugation of partitions reverses dominance order, so $\mu \leq \nu'$ if and only if $\mu' \geq \nu$. In the present setting, we have $\mu = (s-1)^{n-k}$ and $\nu = \overline{\Lambda}$. Since $\overline{\Lambda}_i \leq n-k$ for all $i$, we immediately have $\mu' = (n-k)^{s-1} \geq \overline{\Lambda}$. Thus, $\cB_{(s-1)^{n-k}}^{\overline{\Lambda}} \neq \emptyset$, and $\pi$ is surjective.
\end{proof}

\begin{example}
Let $n$, $\lambda$, $s$, and $T$ be as in Example~\ref{ex:6223unip} and $w=274813$. Note that $T$ is strongly Schubert-compatible.  In this case, $\overline{\Lambda}=(2,2,1,1)$ and $(s-1)^{n-k}=(3,3)$.  We can take as a basis of $\mathbb{C}^K/F^{(w)}_n$ the vectors $f'_5, f'_6, f'_9, f'_{10}, f'_{11}, f'_{12}$, where $f'_i$ denotes the coset $f_i+F^{(w)}_n$.  Then $\cB^{\overline{\Lambda}}_{(s-1)^{n-k}}$ is the Grassmannian of 3-planes in the kernel of $N_T'$, which is spanned by $f'_5, f'_6, f'_{11}, f'_{12}$.

Let our vectors $v_1,\ldots,v_6$, with $v_p\in V_p\setminus V_{p-1}$, be given by the columns of the matrix
\begin{align}\label{eq:VMatrix}
    \begin{bmatrix}
    a & c & d & g & 1 & 0 \\
    1 & 0 & 0 & 0 & 0 & 0 \\
    0 & ab& a & de & 0 & 1 \\
    0 & b & 1 & 0 & 0 & 0 \\
    0 & 0 & 0 & ae & 0 & 0 \\
    0 & 0 & 0 & e & 0 & 0 \\
    0 & 1 & 0 & 0 & 0 & 0 \\
    0 & 0 & 0 & 1 & 0 & 0 \\
    0 & 0 & 0 & 0 & 0 & 0 \\
    0 & 0 & 0 & 0 & 0 & 0 \\
    0 & 0 & 0 & 0 & 0 & 0 \\
    0 & 0 & 0 & 0 & 0 & 0
    \end{bmatrix},
\end{align}
as in~\eqref{eq:CoordinateExampleEq}.

Then the matrix $U$ is given by
\begin{align}
    U=\begin{bmatrix}
    1 & a & 0 & d & 0 & 0 & c & g & 0 & 0 & 0 & 0 \\
    0 & 1 & 0 & 0 & 0 & 0 & 0 & 0 & 0 & 0 & 0 & 0 \\
    0 & 0 & 1 & a & 0 & d & ab& de& 0 & 0 & c & g \\
    0 & 0 & 0 & 1 & 0 & 0 & b & 0 & 0 & 0 & 0 & 0 \\
    0 & 0 & 0 & 0 & 1 & a & 0 & ae& 0 & d & ab& de\\
    0 & 0 & 0 & 0 & 0 & 1 & 0 & e & 0 & 0 & b & 0 \\
    0 & 0 & 0 & 0 & 0 & 0 & 1 & 0 & 0 & 0 & 0 & 0 \\
    0 & 0 & 0 & 0 & 0 & 0 & 0 & 1 & 0 & 0 & 0 & 0 \\
    0 & 0 & 0 & 0 & 0 & 0 & 0 & 0 & 1 & a & 0 & ae \\
    0 & 0 & 0 & 0 & 0 & 0 & 0 & 0 & 0 & 1 & 0 & e \\
    0 & 0 & 0 & 0 & 0 & 0 & 0 & 0 & 0 & 0 & 1 & 0 \\
    0 & 0 & 0 & 0 & 0 & 0 & 0 & 0 & 0 & 0 & 0 & 1
    \end{bmatrix}.
\end{align}

Consider $V'_\bullet\in C_w\cap Y_T$, where $V'_i$ is the span of the first $i$ columns of~\eqref{eq:VMatrix}, and $W'_\bullet\in\mathcal{B}^{(2,2,1,1)}_{(3,3)}$, which is determined entirely by $W'_1$, a 3-dimensional subspace of $f'_5, f'_6, f'_{11}, f'_{12}$. Then the partial flag $\Pi(V'_\bullet,W'_\bullet)$
has $V'_1,\ldots, V'_6$ as its first $n$ parts and $UW'_1$ as its $(n+1)$-th part.

As a specific example, if $W'_1=\vspan\{f'_5, f'_6+2f'_{11}, f'_{12}\}$, then $\Pi(V'_\bullet,W'_\bullet)_7$ would be
\begin{equation}
    V'_6+\vspan\{f'_5, (2c+d)f'_3+(2ab+a)f'_5+(2b+1)f'_6+f'_{11}, gf'_3+def'_5+aef'_9+ef'_{10}+f'_{12} \}.
\end{equation}
\end{example}

\begin{lemma}\label{lem:InjCohomology}
The map on cohomology induced by $\pi$,
\begin{align}
\pi^* : H^*(Y_{n,\la,s}) \to H^*(\cB_\mu^\Lambda),
\end{align}
is injective.
\end{lemma}

\begin{proof}
It suffices to show that $\pi$ satisfies all of the hypotheses of the Relative Affine Paving Lemma (Lemma~\ref{lem:InjectiveCoh}). Indeed, $\pi$ is a continuous map between compact complex algebraic varieties. By Theorem~\ref{thm:AffinePavingY}, for any Schubert compatible $T$, $Y_{n,\la,s,T}$ is paved by the affines spaces $C_{w}\cap Y_{n,\la,s,T}$ for $w$ admissible. By Lemmas~\ref{lem:FiberBundle} and~\ref{lem:PiIsSurj}, $\pi$ is a surjective map such that $\pi^{-1}(C_{w}\cap Y_{n,\la,s,T}) \cong (C_{w}\cap Y_{n,\la,s,T}) \times \cB_{(s-1)^{n-k}}^{\overline{\Lambda}}$. We apply the Relative Affine Paving Lemma taking the $A_{i,j}$ to be the cells $C_w\cap Y_{n,\la,s,T}$ and the $Z_{i,j}$ to be the Spaltenstein varieties $\cB_{(s-1)^{n-k}}^{\overline{\Lambda}}$, each of which has a paving by affine spaces \cite{Brundan-Ostrik}. We showed in Lemma \ref{newlem:TechnicalLemmaUnipotent} that under the isomorphism above, $\pi$ corresponds to the projection onto the first factor, so our proof is complete.
\end{proof}

Recall that $R_{n,\la,s}$ is a graded ring, where we consider $x_i$ to be in degree $2$. The following is our main theorem, which realizes $R_{n,\la,s}$ as the cohomology ring of the variety $Y_{n,\la,s}$.

\begin{theorem}[Theorem~\ref{thm:MainThmIntro}]\label{thm:MainTheorem}
We have an isomorphism of graded rings
\begin{align}
R_{n,\la,s} \cong H^*(Y_{n,\la,s})
\end{align}
given by sending $x_i$ to $-c_1(\widetilde V_{i}/\widetilde V_{i-1})$.
\end{theorem}

\begin{proof}
By Theorem~\ref{thm:Surj}, we have a surjection
\begin{align}\label{eq:MapFromModPowers}
\frac{\bZ[x_1,\dots, x_n]}{\langle x_1^s,\dots, x_n^s\rangle} \twoheadrightarrow H^*(Y_{n,\la,s})
\end{align}
given by sending $x_i$ to $-c_1(\widetilde V_i/\widetilde V_{i-1})$. By Lemma~\ref{lem:IdealContainment}, the cohomology class represented by $e_d(S)$ in $H^*(\cB_\mu^\Lambda)$ for $S\subseteq \{x_1,\dots, x_n\}$ is zero if $d>|S| - p^n_{|S|}(\la)$. By Lemma~\ref{lem:InjCohomology} and naturality of Chern classes, then the cohomology class represented by $e_d(S)$ in $H^*(Y_{n,\la,s})$ is zero as well. Hence, the map \eqref{eq:MapFromModPowers} descends to a map
\begin{align}\label{eq:SurjFromR}
R_{n,\la,s} \twoheadrightarrow H^*(Y_{n,\la,s}).
\end{align}
Since both of these rings are free $\bZ$-modules and have the same Hilbert series by Theorem~\ref{thm:RankGenNLaS}, then \eqref{eq:SurjFromR} is an isomorphism, and the proof is complete.
\end{proof}

By Theorem~\ref{thm:MainTheorem}, we may transfer the action of $S_n$ on $R_{n,\la,s}$ given by permuting the $x_i$ variables to an action of $S_n$ on $H^*(Y_{n,\la,s})$ by permuting the first Chern classes $-c_1(\widetilde V_i/\widetilde V_{i-1})$, making the isomorphism in Theorem~\ref{thm:MainTheorem} an isomorphism of graded $S_n$-modules. It is well known that in the case $n=k$, this action on the cohomology ring of a Springer fiber coincides with Springer's representation tensored with the sign representation.

It is also well known that 
\begin{align}
H^*(\Fl_{(1^n)})(\bC^K)) \cong \left(H^*(\Fl(\bC^K))\right)^{S_1\times \cdots \times S_1\times S_{(n-k)(s-1)}}.
\end{align}
Therefore, there is an action of $S_n$ on $H^*(\Fl_{(1^n)}(\bC^K))$ given by permuting the corresponding first Chern classes on $\Fl_{(1^n)}(\bC^K)$. By naturality of Chern classes, the $S_n$-module structure on $H^*(Y_{n,\la,s})$ is the unique one that makes the surjective map on cohomology
\begin{align}
  H^*(\Fl_{(1^n)}(\bC^K)) \twoheadrightarrow H^*(Y_{n,\la,s})
\end{align}
into an $S_n$-equivariant map. Furthermore, the diagram~\eqref{eq:SnDiagram} of graded $S_n$-modules commutes.

\begin{remark}\label{rmk:s=2}
In the special case where $s=2$, we have $\mu=(1^n,1^{n-k}) = (1^{2n-k})$, so the Spaltenstein variety $\mathcal{B^\nu_\mu}$ that is the source of our map $\pi$ is the Springer fiber for the partition $\Lambda$.  Furthermore, when $w$ is an injective map corresponding to a cell of maximal dimension in $Y_T$, the fibers of the map $\pi$ over $C_w\cap Y_{n,\lambda,s,T}$ as stated in Lemma~\ref{lem:FiberBundle} are points, since $\overline{\Lambda}=(n-k)$ and the Springer fiber for a one row partition is a point.  Hence, the map $\pi$ will be a birational map from certain components of the Springer fiber $\mathcal{B}^\Lambda$ to $Y_{n,\lambda, s,T}$, with other components being contracted in some way.  Note $\Lambda$ has two rows, and 2 row Springer fibers have been extensively studied and their components have explicit descriptions as iterated projective space bundles~\cite{Fung}.  Other papers on 2 row Springer fibers include~\cite{Russell,Khovanov}.

In the special case where $s=2$ and $\lambda=\emptyset$ (of which Example~\ref{ex:Hirzebruch} is the $n=2$ case), $Y_{n,\emptyset,2}$ has one irreducible component, with the unique cell of maximal dimension given by the injective map $w$ that takes the labels of the bottom row in order.  In this case, the endomorphism $N_T|_{V_i}$ of $V_i$ has Jordan type $(i)$, so $\pi$ induces a birational map from the component of the Springer fiber indexed by the Standard Young Tableau with $1,\ldots,n$ in the first row and $n+1,\ldots, 2n$ in the second to $Y_{n,\emptyset,2}$.
\end{remark}

\section{Irreducible components and a generalization of the Springer correspondence}\label{sec:IrreducibleComponents}

In this section, we characterize the irreducible components of $Y_{n,\la,s}$ and show that the number of irreducible components is equal to $\binom{n}{k}\cdot\#\SYT(\la)$ when $s>\ell(\la)$. We then prove that this fact extends to a representation-theoretic statement by giving a generalization of the Springer correspondence to the setting of induced Specht modules.

\subsection{Irreducible components of the \texorpdfstring{$\Delta$}{Delta}-Springer variety}
Fix $k\leq n$, $\la\vdash k$, and $s\geq \ell(\la)$, and set $\Lambda = \Lambda(n,\la,s)$. Given a subspace $W\subseteq \bC^K$ such that $N_\Lambda W\subseteq W$, then $N_\Lambda(W\cap F_k)\subseteq W\cap F_k$. The nilpotent operator $N_\Lambda$ thus induces a nilpotent operator
\begin{align}
    N_\Lambda[W] : F_k/(W\cap F_k) \to F_k/(W\cap F_k).
\end{align}

Let $\mathcal{P}(n,\lambda)$ be the set of fillings of $[\lambda]$ with distinct entries from $[n]$ that decrease left to right along each row and down each column. Note that when $s>\ell(\la)$, 
\begin{equation}
|\mathcal{P}(n,\la)| = \binom{n}{k}\cdot \#\SYT(\la).
\end{equation}
This can be seen explicitly by reversing the labels (via $i \mapsto n-i+1$) to get a row and column increasing tableau on $[\lambda]$ with labels from a $k$-element subset of $[n]$.

Given $S\in \mathcal{P}(n,\la)$, let $S^{[i]}$ be the restriction of $S$ to the cells that contain a label from the set $\{n-i+1,\dots,n-1,n\}$. Furthermore, let $\mathrm{sh}(S^{[i]})$ be the partition obtained by recording the row sizes of $S^{[i]}$ in order from largest to smallest.

For $S\in \mathcal{P}(n,\la)$, define the following locally closed subvariety of $Y_{n,\la,s}$,
\begin{align}
Y_{n,\la,s}^{S} = \{V_\bullet \in Y_{n,\la,s,T} \st N_\Lambda[V_i]\text{ has Jordan type }\mathrm{sh}(S^{[i]})\text{ for all }i\}.
\end{align}
Observe that the definition of $Y_{n,\la,s}^S$ makes sense for any nilpotent operator $N_\Lambda$.  

For the rest of this section, we fix a Schubert-compatible filling $T$ of $[\Lambda]$ and work specifically with $Y_T$.  As the Jordan type of $N_\Lambda[V_i]$ is preserved by conjugation, an argument similar to that of Lemma~\ref{lem:ConjInd} shows that all the statements of this section (except for Lemma~ \ref{lem:IrredSubspaceUnion}, which refers specifically to $T$) hold for arbitrary choices of $N_\Lambda$.

Now that a Schubert compatible $T$ is fixed, given an admissible $w$, let $S(w)$ be the restriction of $\operatorname{IPRD}_T(w)$ to $[\lambda]$. Note that $S(w)$ also depends on $T$, but we suppress this in the notation for convenience. See Figure~\ref{fig:SOfw} for an example of $S(w)$. 

We say two fillings $S_1$ and $S_2$ of $[\la]$ are \textbf{column equivalent} if for each $i$, the set of labels in the $i$-th column of $S_1$ is equal to the set of labels in the $i$-th column of $S_2$.
Equivalently, $S_1$ and $S_2$ are column equivalent if and only if $\sh(S_1^{[i]}) = \sh(S_2^{[ i]})$ for all $i$.

\begin{figure}
  \centering
  \includegraphics[scale=0.45]{Figures/SOfw.pdf}
  \caption{Above, an example of $S(w)$ when $T$ is as shown, where $n=5$, $\la=(2,1)$, $s=3$, and $w=16243$. Below, the partitions $S^{[i]}$ for $1\leq i \leq 5$. \label{fig:SOfw}}
\end{figure}

\begin{lemma}\label{lem:IrredSubspaceUnion}
  For $S\in \mathcal{P}(n,\la)$, we have
  \begin{align}
    Y_{n,\la,s,T}^S = \bigsqcup_{\substack{w\text{ admissible},\\ S(w) \text{ column equivalent to }S}} C_w \cap Y_{n,\la,s,T}.
  \end{align}
  Furthermore, for every admissible $w$, the cell $C_w\cap Y_{n,\la,s,T}$ is contained in $Y_{n,\la,s,T}^S$ for some unique $S \in \mathcal{P}(n,\la)$.
\end{lemma}

  \begin{proof}
    For $w$ admissible, it can be checked that the filling $S$ obtained by sorting each column of $S(w)$ is in $\mathcal{P}(n,\la)$.  Hence $S(w)$ is column equivalent to a unique element of $\mathcal{P}(n,\la)$. Since the cells $C_w\cap Y_{n,\la,s,T}$ partition $Y_{n,\la,s,T}$, it suffices to show the containment $C_w\cap Y_{n,\la,s,T}\subseteq Y_{n,\la,s,T}^S$.

    We proceed by induction on $n$ to show that $C_w\cap Y_{n,\la,s,T}\subseteq Y_{n,\la,s,T}^S$. In the base case $n=0$, the cell $C_w\cap Y_{n,\la,s,T} = Y_{n,\la,s,T}^S$ is a point, where $w$ is the unique injective map between empty sets and $S$ is the empty filling. Let us assume that $n>0$ and the containment holds for $n-1$.

    Let $i\leq s$ such that $w(1) = T(i,\Lambda_i)$, and let $V_\bullet\in C_w\cap Y_T$. Recall from the proof of Lemma~\ref{lem:CellRecursion} that we have an isomorphism
    \begin{equation}
      \Phi : \vspan\{f_{T(h,\Lambda_h)}\st h<i\}\times (C_{\fl_T^{(i)}(w)}\cap Y_{T^{(i)}}) \to C_w \cap Y_T,
    \end{equation}
      with inverse defined by $\Phi^{-1}(V_\bullet) = (v,V_\bullet')$, where $v$ is a vector that spans $V_1$ and
    \begin{equation}
      V_\bullet' = (\psi^{(i)}U^{-1}_{i,v}(V_2)\subseteq\cdots\subseteq\psi^{(i)}U^{-1}_{i,v}(V_n)).
    \end{equation}
   If $i>\ell(\la)$, let $\hat{S}$ be the unique element of $\mathcal{P}(n-1,\la)$ column equivalent to $S(\fl^{(i)}_{T}(w))$; if $\leq \ell(\la)$, let $\hat{S}$ be the unique element of $\mathcal{P}(n-1,\la^{(i)})$ column equivalent to $S(\fl^{(i)}_{T}(w))$.
   By induction, $N_{T^{(i)}}[V_j']$ has Jordan type $\sh(\hat{S}^{[j]})$.
   
     In order to finish the proof, we establish the following string of identities,  where $\operatorname{JT}(-)$ stands for Jordan type,
      \begin{align}
          \operatorname{JT}(N_T[V_j]) &= \operatorname{JT}(N_T[U_{i,v}^{-1}(V_j)])\label{eq:firstJT}\\
          &= \operatorname{JT}(N_{T^{(i)}}[V_{j-1}'])\label{eq:secondJT}\\
          &= \sh(\hat{S}^{[j-1]})\label{eq:thirdJT}\\ \label{eq:fourthJT}
          &= \sh(S^{[j]}).
      \end{align}
      Consider the diagram
      \begin{equation}
          \raisebox{3em}{
          \xymatrix@C=4em{
          \dfrac{F_k}{F_k \cap V_j} \ar[r]^-{U_{i,v}^{-1}} \ar[d]^{N_T[V_j]} &
          \dfrac{F_k}{F_k \cap U_{i,v}^{-1} (V_j)} \ar[r]^-{\psi^{(i)}} \ar[d]^{N_T[U_{i,v}^{-1}(V_j)]} &
          \dfrac{\psi^{(i)}(F_k)}{\psi^{(i)}(F_k) \cap V'_{j-1}}
          \ar[d]^{N_{T^{(i)}}[V'_{j-1}]} \\
          \dfrac{F_k}{F_k \cap V_j} \ar[r]^-{U_{i,v}^{-1}} &
          \dfrac{F_k}{F_k \cap U_{i,v}^{-1} (V_j)} \ar[r]^-{\psi^{(i)}} &
          \dfrac{\psi^{(i)}(F_k)}{\psi^{(i)}(F_k) \cap V'_{j-1}}.
          }}
      \end{equation}
      Observe that $U_{i,v}$ is invertible and preserves $F_k$, while  $\psi^{(i)}$ induces an isomorphism of the given quotient spaces (which follows by considering the cases $i > \ell(\lambda)$ and $i \leq \ell(\lambda)$ separately). It is straightforward to check that both squares of the diagram commute.
      Since the rows of the diagram are isomorphisms, the maps in the columns have the same Jordan type; hence~\eqref{eq:firstJT} and \eqref{eq:secondJT} hold. Identity~\eqref{eq:thirdJT} then follows by induction, and \eqref{eq:fourthJT} follows by column equivalence.
    \end{proof}

    \begin{lemma}\label{lem:IrreducibleNonempty}
      If $s>\ell(\la)$, then $Y^S_T$ is nonempty for any $S\in \mathcal{P}(n,\la)$.  If $s=\ell(\la)$, given $S\in\mathcal{P}(n,\la)$, let $i_S$ be the smallest index such that $i_S$ is not an entry in $S$.  
      Then $Y_T^S$ is nonempty if and only if the subdiagram of $[\lambda]$ consisting of the cells of $S$ labeled by $1,\dots, i_S-1$ contains the entire last row of $[\lambda]$.
    \end{lemma}

    \begin{proof}
        It can be checked that there exists an admissible $w$ such that $S(w)$ is column equivalent to $S$ if and only if the conditions in the statement of the lemma hold. In the case $s=\ell(\la)$, the extra condition that the subdiagram of $[\lambda]$ consisting of the cells of $S$ labeled by $1,\dots, i_S-1$ contains the entire last row of $[\lambda]$ corresponds to the fact that if $w$ is admissible, then the smallest label $i$ not in $S(w)$ must be greater than the smallest entry in the first column of $S(w)$, which means that the last row of the unique element of $\mathcal{P}(n,\la)$ column equivalent to $S(w)$ must be filled with entries smaller than the label $i$. Therefore, by Lemma~\ref{lem:IrredSubspaceUnion}, these are exactly the conditions for when $Y_T^S$ is nonempty.
    \end{proof}

    \begin{lemma}\label{lem:IrreducibleDimension}
      Let $S\in \mathcal{P}(n,\la)$. When $Y_T^S$ is nonempty, it is smooth and connected of dimension $n(\la) + (n-k)(s-1)$.
    \end{lemma}

    \begin{proof}
      Let $S\in \mathcal{P}(n,\la)$ be an element satisfying the conditions in the statement of Lemma~\ref{lem:IrreducibleNonempty}. Let $\hat{S}$ be the result of subtracting $1$ from each entry of $S^{[n]}$. Specifically, if $S$ contains the label $1$, and $(i,j)$ is the cell of $[\la]$ containing it, then $\hat{S}\in \mathcal{P}(n-1,\la^{(i)})$ is the result of deleting the cell $(i,j)$ of $S$ and then subtracting $1$ from each entry. Otherwise, let $(i,j) = (s,n-k)$, and  $\hat{S}\in \mathcal{P}(n-1,\la)$ is the result of subtracting $1$ from each entry of $S$. It can be checked that $\hat{S}$ satisfies the conditions in the statement of Lemma~\ref{lem:IrreducibleNonempty} with respect to $T^{(i)}$, so $Y_{T^{(i)}}^{\hat{S}}$ is nonempty.
      
      We have a surjective map $\pi$,
      \begin{align}
        Y_T^S \xrightarrow{\pi} \bP(\vspan\{f_{T(h,\Lambda_h)} \st h\leq i\}) \setminus \bP(\vspan\{f_{T(h,\Lambda_h)}\st h\leq \Lambda_{j+1}'\}),
      \end{align}
      defined by sending $V_\bullet$ to $V_1$. Let
      \begin{align}
        U \coloneqq \bP(\vspan\{f_{T(h,\Lambda_h)} \st h\leq i\}) \setminus \bP(\vspan\{f_{T(h,\Lambda_h)}\st h < i\}).
      \end{align}
      
      We prove the lemma by induction on $n$. In the base case $n=0$, we have $k=0$ and $\la = \emptyset$. Therefore, there is a unique $Y_T^S$ which is a single point, which agrees with the dimension formula since $n(\la)+(n-k)(s-1)=0$ in this case, so the base case holds. In the inductive step, let $n>0$ and assume by induction on $n$ that $Y_{T^{(i)}}^{\hat{S}}$ is smooth and connected with dimension $n(\la^{(i)})+(n-k)(s-1)$ for $i\leq \ell(\la)$ and $n(\la)+(n-k-1)(s-1)$ for $i>\ell(\la)$.  By  Lemma~\ref{lem:IrredSubspaceUnion}, 
      \begin{align}
          \pi^{-1}(U)=\bigsqcup_{w(1)=i} C_w\cap Y_T,
      \end{align}
      so by Remark~\ref{rmk:CombinedIso}
      there is a combined isomorphism
      \begin{align}
        \Phi : U\times Y_{T^{(i)}}^{\hat{S}} \to \pi^{-1}(U).
      \end{align}
      Thus, by the inductive hypothesis, the open subset $\pi^{-1}(U)$ of $Y_T^S$ is smooth and connected. Its dimension is
      \begin{align}
        \dim(U) + \dim(Y_{T^{(i)}}^{\hat{S}}) &= (i-1) + \begin{cases} n(\la^{(i)}) + (n-k)(s-1) & i\leq \ell(\la)\\ n(\la) + (n-k-1)(s-1) & i>\ell(\la)\end{cases}\\
        &= n(\la) + (n-k)(s-1).
      \end{align}
      
      Given some $p$ with $\Lambda_{j+1}'+1\leq p \leq i$, we have an open subset
      \begin{align}
        U_{p} = \bP(\vspan\{f_{T(h,\Lambda_h)}\st h\leq i\})\setminus \bP(\vspan\{f_{T(h,\Lambda_h)}\st h\leq i, h\neq p\})
      \end{align}
      of the codomain of $\pi$. By applying a change of basis corresponding to a permutation of the rows of $T$, we see that $U\cong U_p$ and $\pi^{-1}(U) \cong \pi^{-1}(U_p)$. Thus, $\pi^{-1}(U_p)$ is smooth and connected of the same dimension for all $p$. Since $Y_T^S$ is covered by smooth connected open subsets of dimension $n(\la)+(n-k)(s-1)$, it follows that $Y_T^S$ is smooth of dimension $n(\la)+(n-k)(s-1)$. Since each $\pi^{-1}(U_{p})$ is connected and the intersection of these subspaces is nonempty, $Y_T^S$ is connected.
    \end{proof}
    
  \begin{example}
    Let $n=4$, $\la = (2,1)$, $s=3$, and
    \[
    \includegraphics[scale=0.45]{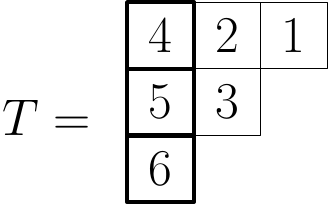}.
    \]
    Then $Y_{n,\la,s,T}$ decomposes into the following smooth connected subspaces of dimension $n(\la) + (n-k)(s-1)=3$, where $C'_w$ stands for the cell $C_w \cap Y_{n,\la,s,T}$.
    \begin{alignat*}{4}
        Y_T^{\scriptsize\young(21,3)} &= C'_{1234}\cup C'_{1235}\cup C'_{1236}\cup C'_{1324}\cup C'_{1325}\cup C'_{1326}, &\,\,&& Y_T^{\scriptsize\young(42,1)} &= C'_{3152} \cup C'_{3162},\\
        Y_T^{\scriptsize\young(21,4)} &=C'_{1243} \cup C'_{1263} \cup C'_{1352}\cup C'_{1362}, &&& Y_T^{\scriptsize\young(43,1)} &= C'_{3512} \cup C'_{3612},\\
        Y_T^{\scriptsize\young(41,3)} &= C'_{1623} \cup C'_{1632}, &&& Y_T^{\scriptsize\young(42,3)} &= C'_{6123} \cup C'_{6132},\\
        Y_T^{\scriptsize\young(32,1)} &= C'_{3124} \cup C'_{3125} \cup C'_{3126}, &&& Y_T^{\scriptsize\young(43,2)} &= C'_{6312}.
    \end{alignat*}
    \end{example}

\begin{theorem}[Theorem~\ref{thm:DimensionOfYIntro}] \label{thm:irreducible-components}
The space $Y_{n,\la,s,T}$ is equidimensional of complex dimension $n(\la) + (n-k)(s-1)$. In particular, the closed subvarieties $\overline{Y_{n,\la,s,T}^{S}}$ for which $Y_{n,\la,s,T}^{S}$ is nonempty (as described in~Lemma \ref{lem:IrreducibleNonempty}) form a complete set of irreducible components. In the case $s>\ell(\la)$, there are $\binom{n}{k}\cdot \#\SYT(\la)$ many irreducible components.
\end{theorem}

\begin{proof}
By Lemma~\ref{lem:IrreducibleDimension}, the nonempty subspaces $Y_{n,\la,s,T}^S$ are smooth and connected of dimension $n(\la)+(n-k)(s-1)$. Therefore, $Y_{n,\la,s,T}^S$ is irreducible. Since the subspaces $Y_{n,\la,s,T}^S$ partition the space $Y_{n,\la,s,T}$, their closures form a complete set of irreducible components. When $s>\ell(\la)$, the irreducible components are thus in bijection with $\mathcal{P}(n,\la)$, which has cardinality $\binom{n}{k}\cdot \#\SYT(\la)$.
\end{proof}

\subsection{Generalization of the Springer correspondence}
In the case $s>\ell(\la)$, the irreducible components are naturally indexed by Standard Young Tableaux on $\la \cup (n-k)$.  This indexing of irreducible components extends to a realization of the top cohomology group of $Y_{n,\lambda,s}$ as a $S_n$-module, generalizing Springer's theorem that the top cohomology group of a Springer fiber is a Specht module.

\begin{theorem}[Theorem~\ref{thm:GenSpringerCorrespondence}]\label{thm:GenSpringerCorrespondence2}
Let $d = \dim(Y_{n,\la,s}) = n(\la) + (n-k)(s-1)$, and consider $S_k$ as the subgroup of $S_n$ permuting the elements of $[k]$. For $s>\ell(\la)$, we have an isomorphism of $S_n$-modules
\begin{align}
H^{2d}(Y_{n,\la,s};\bQ) \cong \mathrm{Ind}\!\uparrow_{S_k\times S_{n-k}}^{S_n} (S^\la),
\end{align}
where $S_k\times S_{n-k}$ acts on $S^\la$ by its usual action of $S_k$
(and $S_{n-k}$ acts trivially).
For $s=\ell(\la)$, we have
\begin{align}\label{eq:SLengthLambdaIso}
    H^{2d}(Y_{n,\la,s};\bQ) \cong S^{\Lambda/(n-k)^{s-1}},
\end{align}
the Specht module of skew shape $\Lambda/(n-k)^{s-1}$.
\end{theorem}

Before we give the proof of Theorem~\ref{thm:GenSpringerCorrespondence2}, we recall a formula for the graded Frobenius characteristic of $R_{n,\la,s}$ proved in~\cite{GriffinOSP}. A \textbf{partial row-decreasing filling} $\varphi$ of $[\Lambda]$ is a filling of a subset of the cells of $[\Lambda]$ with positive integers (allowing repeated labels) such that the labels in each row are right justified and weakly decrease from left to right and such that all cells of $[\la]$ are filled. Let $\mathrm{PRD}_{n,\la,s}$ be the set of partial row-decreasing fillings of $[\Lambda]$ with a total of $n$ filled cells. 

Recall our convention that $(i,j)\in [\Lambda]$ is the cell in the $i$-th row from the top and the $j$-th column from the left. For $\varphi\in \mathrm{PRD}_{n,\la,s}$ and $(i,j)\in [\Lambda]$ a filled cell, let $\varphi(i,j)$ be the label of $\varphi$ in that cell. 
Given a weak composition $\alpha$, we say that $\varphi$ has \textbf{content $\alpha$} if for each $i$, the letter $i$ appears $\alpha_i$ many times as a label in $\varphi$.

\begin{figure}
    \centering
    \includegraphics[scale=0.45]{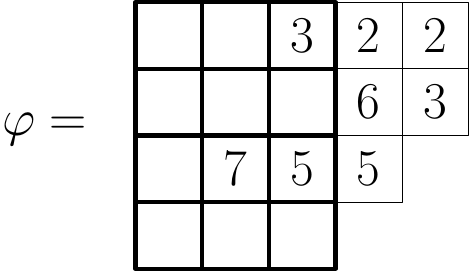}
    \caption{A partial row-decreasing filling $\varphi\in \mathrm{PRD}_{8,(2,2,1),4}$ of type $\alpha=(0,2,2,0,2,1,1)$ with $\inv(\varphi)=8$.}
    \label{fig:ExInversions}
\end{figure}

\begin{definition}
Given $\varphi\in \mathrm{PRD}_{n,\la,s}$, an \textbf{inversion} of $\varphi$ is one of the following,
\begin{enumerate}
\item  [(I1)] A pair of cells $(i,j),(i',j)\in [\la]$ such that $i<i'$ and $\varphi(i,j) > \varphi(i',j)$,
\item  [(I2)] A pair of cells $(i,j),(i',j-1)\in [\la]$ such that $i'<i$ and $\varphi(i,j) > \varphi(i',j-1)$,
\item  [(I3)] A cell $(i,n-k+1)\in [\la]$ in the first column of $[\la]$ together with a filled cell $(i',j')\in [\Lambda]\setminus [\la]$ such that $i' < i$ and $\varphi(i,n-k+1) > \varphi(i',j')$,
\item  [(I4)] A pair $(i,c)$, where $i$ is an integer and $c = (i',j')$ is a filled cell of $[\Lambda]\setminus [\la]$ such that $1\leq i<i'$.
\end{enumerate}
Let $\inv(\varphi)$ be the number of inversions of $\varphi$. 
\end{definition}

See Figure~\ref{fig:ExInversions} for an example of $\varphi\in \mathrm{PRD}_{8,(2,2,1),4}$. In this case, $\varphi$ has the following $8$ inversions: The pair $(2,4),(3,4)$ of type (I1), the pair $(2,5),(1,4)$ of type (I2), the cell $(2,4)$  in the first column of $[\la]$ with the filled cell $(1,3)$ of type (I3), the cell $(3,4)$ in the first column of $[\la]$ with the filled cell $(1,3)$ of type (I3), and the pairs $(1,(3,2))$, $(2,(3,2))$, $(1,(3,3))$, and $(2,(3,3))$ of type (I4).

\begin{remark}
The partial row-decreasing fillings defined here are a simple variation of the extended column-increasing fillings defined in~\cite{GriffinOSP}. There is an obvious bijection between $\mathrm{PRD}_{n,\la,s}$ defined here and $\mathrm{ECI}_{n,\la,s}$ defined in~\cite{GriffinOSP}: Given $\varphi\in \mathrm{PRD}_{n,\la,s}$, rotate the filling by 90 degrees counterclockwise, and delete all unfilled cells. The inversion statistic $\mathrm{inv}$ defined here is also a simple variation of the $\mathrm{inv}$ statistic in \cite{GriffinOSP}, which in turn is a variation on the \textit{coinversion} statistic on ordered set partitions originally defined by Rhoades, Yu, and Zhao in~\cite{Rhoades-Yu-Zhao}.
\end{remark}

We have the following formula for the graded Frobenius characteristic of $R_{n,\la,s}$. We state it in terms of partial row-decreasing fillings and powers of $q^2$ (due to our convention that $x_i$ has degree $2$).
\begin{theorem}[{\cite[Theorem 5.13]{GriffinOSP}}]\label{thm:grFrobTheorem}
  We have
  \begin{equation}
    \Frob(R_{n,\la,s}^\bQ;q) = \sum_{\varphi \in \mathrm{PRD}_{n,\la,s}} q^{2\,\inv(\varphi)}\bx^\varphi,
  \end{equation}
  where $\bx^\varphi$ is the monomial such that the power of $x_i$ is the number of times $i$ appears as a label in $\varphi$.
\end{theorem}

\begin{proof}[Proof of Theorem~\ref{thm:GenSpringerCorrespondence2}]
  By Theorem~\ref{thm:MainTheorem}, we have $H^*(Y_{n,\la,s};\bQ) \cong R_{n,\la,s}^\bQ$ as graded $S_n$-modules.
The case when $s>\ell(\lambda)$ follows immediately by combining this
isomorphism with \cite[Equation 3.3.27]{GriffinThesis}, which says
that the top degree component of $R_{n,\lambda,s}^{\bQ}$ is isomorphic
as an $S_n$-module to $\mathrm{Ind}\!\uparrow_{S_k\times
  S_{n-k}}^{S_n}\!(S^\lambda)$, where $S_k$ acts in the usual way on
$S^\lambda$ and $S_{n-k}$ acts trivially.

Let us now assume $s=\ell(\lambda)$. Recall $d = \dim_\bC Y_{n, \lambda, s} = n(\lambda) + (n-k)(s-1).$ Combining Theorem~\ref{thm:MainTheorem} and Theorem~\ref{thm:grFrobTheorem}, we have that
\begin{align}\label{eq:FrobCohMonomial}
    \Frob(H^{2d}(Y_{n,\lambda,s};\bQ)) = \sum_{\substack{\varphi \in \mathrm{PRD}_{n,\lambda,s},\\ \mathrm{inv}(\varphi)=d}} \mathbf{x}^\varphi.
\end{align}
We show that the right-hand side of \eqref{eq:FrobCohMonomial} is
equal to the skew Schur function $s_{\Lambda/(n-k)^{\ell(\la)-1}}(\bx)$.

Let $\alpha = (\alpha_1,\dots, \alpha_n)$ be a weak composition of $n$ into $n$ parts.
The coefficient of $\bx^\alpha = \prod_i x_i^{\alpha_i}$ in the right-hand side of \eqref{eq:FrobCohMonomial} is the number of $\varphi\in \mathrm{PRD}_{n,\la,s}$ with content $\alpha$ such that $\mathrm{inv}(\varphi) = d$. It follows from \cite{HHL} that the total number of (I1) and (I2) inversions of $\varphi$ is $n(\la)$ if and only if the entries of $\varphi$ in $[\lambda]$ decrease down each column (and $n(\la)$ is the maximum possible total). Furthermore, each of the $n-k$ cells of $[\Lambda]\setminus[\la]$ can be part of a maximum of $s-1$ inversions of type (I3) and (I4), and if any (I4) inversion occurs, then the entries of $\varphi$ in $[\lambda]$ do not decrease down each column. Therefore, $\inv(\varphi)=d$ if and only if the entries of $\varphi$ in $[\lambda]$ decrease down each column and each filled cell of $[\Lambda]\setminus[\la]$ is in row $s=\ell(\la)$.

Given such a $\varphi\in \mathrm{PRD}_{n,\la,s}$ with $\inv(\varphi)=d$, 
we define a labeling $T$ of the Young diagram of skew shape $\Lambda/(n-k)^{\ell(\lambda)-1}$ by deleting all empty cells of $[\Lambda]$ and replacing each label $j$ with $n+1-j$. Since the labels of $\varphi$ weakly decrease from left to right along each row and strictly decrease down each column, $T$ is semi-standard, and its content is $(\alpha_n,\alpha_{n-1},\dots, \alpha_1)$. Moreover, $\varphi$ can easily be reconstructed from $T$.

This shows that the coefficient of $\bx^\alpha$ in the right-hand side
of \eqref{eq:FrobCohMonomial} is equal to the coefficient of
$\bx^{(\alpha_n,\alpha_{n-1},\dots, \alpha_1)}$ in
$s_{\Lambda/(n-k)^{\ell(\la)-1}}(\bx)$. Since both the skew Schur
function $s_{\Lambda/(n-k)^{\ell(\la)-1}}(\bx)$ and the right-hand side of \eqref{eq:FrobCohMonomial} are symmetric, it follows that they are equal.
\end{proof}

\section{The \texorpdfstring{$\Delta$}{Delta}-Springer ind-variety and the scheme of diagonal \\rank-deficient matrices}\label{sec:IndVariety}

In this section, we construct an ind-variety $Y_{n,\lambda}$ as the direct limit of the spaces $Y_{n,\la,s}$ as $s\to \infty$. We then prove the cohomology ring of $Y_{n,\lambda}$ is isomorphic to the coordinate ring of the scheme-theoretic intersection of an Eisenbud--Saltman rank variety~\cite{Eisenbud-Saltman} with diagonal matrices, generalizing a similar characterization of the cohomology ring of the Springer fiber due to De Concini and Procesi~\cite{dCP}.

\subsection{The ind-variety \texorpdfstring{$Y_{n,\lambda}$}{Yn,lambda}}
For any integer $0\leq k \leq n$ and any partition $\lambda\vdash k$, define $Y_{n,\la}$ as follows.  Let $\bC^\infty$ be the countably-infinite dimensional $\bC$-vector space with basis $\{f_1, f_2, f_3,\ldots\}$, and let $N$ be a nilpotent operator on $\bC^\infty$ with Jordan type 
\begin{equation}
(n-k+\la_1, \dots, n-k+\la_{\ell(\la)}, n-k,n-k,\dots).
\end{equation}
Without loss of generality, we assume that the $f_i$'s are the generalized eigenvectors for $N$.  Furthermore, we assume that, for all $s \geq \ell(\la)$, $N$ fixes the subspace $\vspan{\{f_1,\ldots,f_{k+(n-k)s}\}}$ and the restricted endomorphism $N|_{\vspan{\{f_1,\ldots,f_{k+(n-k)s}\}}}$
has Jordan type
\begin{equation}
(n-k+\la_1, \dots, n-k+\la_{\ell(\la)}, (n-k)^{s-\ell(\la)}).
\end{equation}
Note that $\im(N^{n-k})$ has dimension $k$.

We define the \textbf{$\Delta$-Springer ind-variety}
\begin{align}
    Y_{n,\la}\coloneqq \{V_\bullet\in \Fl_{(1^n)}(\bC^\infty)\st NV_i\subseteq V_{i-1} \text{ for }i\leq n\text{ and }\im(N^{n-k})\subseteq V_n\}.
\end{align}
By our assumptions on $N$, we have closed embeddings
\begin{equation}
Y_{n,\la,\ell(\la)}\subseteq Y_{n,\la,\ell(\la)+1}\subseteq\cdots\subseteq Y_{n,\la,s}\subseteq\cdots.
\end{equation}
We consider the direct limit
\begin{equation}
Y_{n,\la} = \bigcup_{s \geq \ell(\la)} Y_{n, \lambda, s} \cong \varinjlim_s Y_{n,\la,s}.
\end{equation}
(Note that the topology of $Y_{n,\la}$ as an ind-variety is the same as the subspace topology as a subset of $\Fl_{(1^n)}(\bC^\infty)$ since it is a closed ind-subvariety of $\Fl_{(1^n)}(\bC^\infty)$; see \cite[Ch. 4]{Kumar}.)

We continue to abuse notation and use $\widetilde V_i$ to denote the tautological rank $i$ vector bundle on $Y_{n,\la}$, which is the subspace of $\Fl_{(1^n)}(\bC^{\infty})\times \bC^\infty$ whose fiber over $V_\bullet$ is $V_i$.

Recall the graded rings $R_{n,\la}$ defined in Section~\ref{sec:Background}, which in general have infinitely many nonzero graded components.  They are defined by
\begin{align}
I_{n,\la} &\coloneqq \langle e_d(S) \st S\subseteq \{x_1,\dots, x_n\},\, d>|S| - \la_n' - \cdots - \la_{n-|S|+1}'\rangle,\\
R_{n,\la} &\coloneqq \bZ[x_1,\dots,x_n]/I_{n,\la}.
\end{align}
Furthermore, let $R_{n,\la}^\bQ$ be the quotient of $\bQ[x_1,\dots, x_n]$ by the $\bQ$-span of $I_{n,\la}$.

\begin{theorem}\label{thm:InfiniteIso}
We have $H^*(Y_{n,\la})\cong R_{n,\la}$ as graded rings, such that the
cohomology class $-c_1(\widetilde V_i/\widetilde V_{i-1})$ is identified with $x_i$.
\end{theorem}

\begin{proof}
By Theorem~\ref{thm:MainTheorem}, we have $H^*(Y_{n,\lambda,s})\cong R_{n,\la,s}$.  From the definitions of $R_{n,\la}$ and $R_{n,\la,s}$, it can be checked that
\begin{align}
R_{n,\la} \cong \varprojlim_s R_{n,\la,s} \cong \varprojlim_s H^*(Y_{n,\la,s}),
\end{align}
where the inverse limit is taken in the category of graded rings. Since each $Y_{n,\la,s}$ is a complex variety with an affine paving, we have
\begin{align}
  \varprojlim_s H^*(Y_{n,\la,s}) \cong H^*(\varinjlim_s Y_{n,\la,s}) = H^*(Y_{n,\la});
\end{align}
see for example~\cite[Lemma 7.2]{Pawlowski-Rhoades}. By naturality of Chern classes, the first Chern class $-c_1(\widetilde V_i/\widetilde V_{i-1})$ in $H^*(Y_{n,\la})$ corresponds to its associated first Chern class in $H^*(Y_{n,\la,s})$, which is in turn identified with the variable $x_i$. This completes the proof.
\end{proof}

\subsection{The scheme of diagonal rank-deficient matrices}

Let $\fgl_n$ be the space of $n\times n$ matrices over $\bQ$. Let $x_{i,j}$  for $1\leq i,j \leq n$ be the coordinate functions corresponding to the entries of an $n\times n$ matrix. Then the coordinate ring of $\fgl_n$ is $\bQ[\fgl_n] = \bQ[x_{i,j}]$.  

For $\la\vdash n$, let $O_\la\subseteq \fgl_n$ be the conjugacy class of nilpotent $n\times n$ matrices over $\bQ$ whose Jordan canonical form has block sizes recorded by $\la$. Let $\overline{O}_\la$ be the closure of $O_\la$ in $\fgl_n$. The set of diagonal matrices $\ft$ is the variety defined by the ideal
\begin{align}
I(\ft) = \langle x_{i,j} \st i\neq j\rangle.
\end{align}
The \textbf{scheme-theoretic intersection} of the varieties $\overline{O}_\la$ and $\ft$ is the affine scheme whose coordinate ring is defined by the sum of the defining ideals of $\overline{O}_\la$ and $\ft$,
\begin{align}\label{eq:CoordinateRing}
\bQ[\overline{O}_\la \cap \ft] \coloneqq \frac{\bQ[x_{i,j}]}{I(\overline{O}_\la) + I(\ft)}.
\end{align}
The symmetric group $S_n$ of permutation matrices acts by conjugation on both $\overline{O}_\la$ and $\ft$, so we have an action of $S_n$ on $\bQ[\overline{O}_\la\cap \ft]$. Observe that the variables $x_{i,i}$ generate this coordinate ring. Re-indexing the generators $x_{i,i}$ of $\bQ[\overline{O}_\la \cap \ft]$ as $x_i$, $S_n$ acts by permuting the $x_i$ variables. We consider $\bQ[\overline{O}_\la\cap \ft]$ as a graded ring and graded $S_n$-module where $x_{i,i}$ is declared to be in degree $2$.

Motivated by work of Kostant~\cite{Kostant} on the coinvariant algebra, Kraft~\cite{Kraft} conjectured that the coordinate ring~\eqref{eq:CoordinateRing} is isomorphic to the cohomology ring of a Springer fiber. De Concini and Procesi~\cite{dCP} proved Kraft's conjecture. Tanisaki~\cite{Tanisaki} then simplified the arguments of De Concini and Procesi and further proved that these rings have the explicit presentation as the quotient ring $R_{\la} = R_{n, \la, \ell(\la)}$.

\begin{theorem}[\cite{dCP,Tanisaki}]
There are isomorphisms of graded rings and graded $S_n$-modules
\begin{align}
H^*(\mathcal{B}^\lambda;\bQ) \cong R_\la^\bQ \cong \bQ[\overline{O}_{\la'}\cap \ft],
\end{align}
where the $S_n$ action on $H^*(\mathcal{B}^\lambda;\bQ)$ permutes the first Chern classes $-c_1(\widetilde V_i/\widetilde V_{i-1})$ and corresponds to Springer's representation tensored with the sign representation.
\end{theorem}

In~\cite{Eisenbud-Saltman}, Eisenbud and Saltman study varieties generalizing the varieties $\overline{O}_{\la}$. These varieties are the main focus of this section.

\begin{definition}
Let $k\leq n$, and let $\la\vdash k$. The \textbf{Eisenbud--Saltman rank variety} is the variety
\begin{align}
\overline{O}_{n,\la} &\coloneqq \{ X\in \fgl_n\st \dim\ker X^d \geq \la_1' + \cdots +\la_d', \, d=1,2,\dots, n\}\\
&= \{X\in \fgl_n \st \rk(X^d)\leq (n-k)+\la'_{d+1}+\cdots +\la'_n,\, d=1,2,\dots, n\}.
\end{align}
\end{definition}
The variety $\overline{O}_{n,\la}$ is the same as $X_r$ defined in \cite{Eisenbud-Saltman}, where $r$ is the rank function given by $r(d) = (n-k)+\la'_{d+1}+\cdots + \la'_n$. Note that the ideal cutting out $\overline{O}_{n,\la}$ is not in general generated by all $(r(d)+1)$-minors of $X^d$ for $d=1,\dots,n$ but is instead the radical of that ideal.  When $n=k$, we have $\overline{O}_{n,\la} = \overline{O}_\la$. When $n>k$, then $\overline{O}_{n,\la}$ contains matrices which are not nilpotent.  In particular, the variety $\overline{O}_{n,\la}$ contains all block diagonal matrices of the form
\begin{align}\label{eq:JordanForm}
X_\la \oplus A_{n-k} = \left[\begin{array}{c|c} X_{\la} & 0 \\ \hline 0 & A_{n-k}\end{array} \right],
\end{align}
where $X_\la\in \overline{O}_\la\subseteq \fgl_k$ and $A_{n-k}\in \fgl_{n-k}$, as well as any matrix that is conjugate to a matrix in this form.

 By combining~\cite[Corollary 6.4]{GriffinOSP} and Theorem~\ref{thm:InfiniteIso}, we have the following extension of the result of De Concini and Procesi.

\begin{corollary}
We have a string of isomorphisms of graded rings and graded $S_n$-modules,
\begin{align}
H^*(Y_{n,\la};\bQ) \cong R_{n,\la}^\bQ \cong \bQ[\overline{O}_{n,\la'}\cap \ft],
\end{align}
where the right-hand side is the coordinate ring of the scheme-theoretic intersection, and $S_n$ acts on the cohomology ring by permuting the first Chern classes $-c_1(\widetilde V_i/\widetilde V_{i-1})$.
\end{corollary}

Let 
\begin{align}
\mathrm{PRD}_{n,\la} &\coloneqq \bigcup_{s\geq \ell(\la)} \mathrm{PRD}_{n,\la, s},
\end{align}
where we identify $\varphi\in \mathrm{PRD}_{n,\la, s}$ with the partial row-decreasing filling in the set $\mathrm{PRD}_{n,\la, s+1}$ obtained by appending an empty $(s+1)$-th row to $\varphi$. Observe that for each $\varphi \in \mathrm{PRD}_{n,\la, s}$, the statistic $\inv(\varphi)$ does not depend on the parameter $s$. Hence, we may consider $\inv$ to be a statistic on elements of $\mathrm{PRD}_{n,\la}$.

We have the following corollary of \cite[Theorem 6.6]{GriffinOSP} and Theorem~\ref{thm:InfiniteIso}.

\begin{corollary}
For any $k\leq n$ and $\la\vdash k$,
\begin{equation}
\Frob(H^*(Y_{n,\la};\bQ);q) = \Frob(R_{n,\la}^\bQ;q) = \Frob(\bQ[\overline{O}_{n,\la'} \cap \ft];q)  = \sum_{\varphi \in \mathrm{PRD}_{n,\la}} q^{2\,\inv(\varphi)} \bx^\varphi.
\end{equation}
\end{corollary}

\section{Future Work}\label{sec:FutureWork}

We conclude with a number of questions. 
First, the $S_n$-module structure of the cohomology of a Springer fiber, $H^*(\cB^\lambda;\bQ)$, has several constructions.  The construction of Borho and MacPherson~\cite{Borho-MacPherson} uses the Grothendieck--Springer resolution and intersection cohomology, and Brundan and Ostrik~\cite{Brundan-Ostrik} give a similar construction for the cohomology of Spaltenstein varieties.

\begin{question}\label{q:Sn-rep}
Can the $S_n$-module structure on $H^*(Y_{n,\lambda,s};\bQ)$ be directly realized as a specialization of an $S_n$-action on an intersection cohomology complex coming from a generalization of the Grothendieck--Springer resolution?
\end{question}

\begin{remark}
We note that Question~\ref{q:Sn-rep} has recently been answered by Gillespie and the first author in \cite{GG-skew}, where they prove that each $Y_{n,\la,s}$ is equal to one of the varieties studied in \cite{Borho-MacPherson}. They then use this connection to prove new formulas for the symmetric function in the Delta Conjecture at $t=0$.
\end{remark}

Springer fibers have a natural description that works for an arbitrary semisimple reductive algebraic group.  Hence we ask the following.

\begin{question}
Can the varieties $Y_{n,\la,s}$ be generalized to arbitrary Lie type, or at least to the classical types?
\end{question}

Our affine paving does not form a cell decomposition, as the closure of a cell is not always a union of cells.  Hence we can ask the following.

\begin{question}
What are the cell closures for the paving of $Y_{n,(1^k),k}$?  Is there a nice description of the poset that describes when one cell is contained in the closure of another?  Furthermore, consider the directed graph whose vertices correspond to cells and there is an edge from vertex $C$ to $C'$ if the cell $C'$ has a point in the closure of $C$.  This directed graph is acyclic since it is a subgraph of Bruhat order.  Is there a nice description of the poset generated by this directed graph?
\end{question}

Note the answer to this question is not known even for Springer fibers in general, but it may be more tractable in certain special cases.

The following is also known for the Springer fiber in some special cases but not in general.

\begin{question}
 The space $Y_{n,\la,s}$ is singular in general because it is connected and has many irreducible components. Under what conditions are all of the irreducible components smooth?  In particular, are all the irreducible components of $Y_{n,(1^k),k}$ smooth?
\end{question}

One can also ask about other properties of the irreducible components, such as whether they are normal or Cohen-Macaulay.

We have some additional questions in the case $s=2$.

\begin{question}
As noted in Remark~\ref{ex:Hirzebruch}, the special case $Y_{n,\emptyset,2}$ appears in the work of Cautis and Kamnitzer~\cite{Cautis-Kamnitzer} and Russell~\cite{Russell} on connections between Springer fibers and knot homology. Can the homology of all $\Delta$-Springer fibers in the $s=2$ case $Y_{n,\la,2}$ also be described using skein-theoretic relations?
\end{question}

\begin{question}
As noted in Remark~\ref{rmk:s=2}, $\pi$ induces a birational map from a union of components of a 2 row Springer fiber to $Y_{n,\la,s}$.  Is this map an isomorphism?
\end{question}

One way to answer this question would be to figure out enough details of the cells in these irreducible components of the Springer fiber to determine if the map $\pi$ restricted to these cells is an isomorphism.  One can also generalize this question, as a similar argument shows that $\pi$ induces a birational map from a union of certain components of a Spaltenstein variety to $Y_{n,\la,s}$ for arbitrary $s$.

Finally, we have a more speculative question.

\begin{question}
  The variety $Y_{n,\la,s}$ is simultaneously the projected image of a Spaltenstein variety and a closed subvariety of a Steinberg variety,
\begin{equation}
    \cB_{(1^n,s-1,\dots,s-1)}^\Lambda \twoheadrightarrow Y_{n,\la,s}\hookrightarrow \{V_\bullet\in \Fl_{(1^n,(n-k)(s-1))}(\bC^K) \st N_\Lambda V_i\subseteq V_i \text{ for all }i\}.
\end{equation}
Both Steinberg varieties and Spaltenstein varieties appear in Borho and MacPherson's study of partial resolutions of the nilpotent cone~\cite{Borho-MacPherson}, and Spaltenstein varieties have connections to representations of the general linear groups~\cite{Braverman-Gaitsgory}, crystals~\cite{Malkin} and quiver varieties~\cite{Nakajima1,Nakajima2}. Do the varieties $Y_{n,\la,s}$ have any applications in these settings?
\end{question}

\section{Acknowledgements}

We are grateful to Sara Billey, Maria Gillespie, Erik Insko, Abel Lacabanne, Isabella Novik, Julia Pevtsova, Martha Precup, Claudio Procesi, Brendon Rhoades,  Pedro Vaz, Arik Wilbert, and Andy Wilson for helpful conversations and comments. Thanks especially to Leonardo Mihalcea for useful conversations about Hirzebruch surfaces. The first author would like to thank the Institute for Computational and Experimental Research Mathematics (ICERM) for their hospitality, where he was a resident in Spring 2021 during which part of the current work was completed.



\bibliographystyle{hsiam}
\bibliography{Springer}





\end{document}